\documentclass[12pt,a4paper, oneside]{amsart}
\usepackage{geometry}

\usepackage[sc]{mathpazo}
\usepackage[T1]{fontenc}
\usepackage{amsmath,amssymb}
\usepackage{xcolor}
\usepackage{parskip}
\usepackage{afterpage}
\usepackage{url}
\usepackage{amsthm}
\usepackage{emptypage}
\usepackage{tikz-cd}
\usepackage{cleveref}

    \usepackage{etoolbox}
    \patchcmd{\section}{\scshape}{\large\bfseries}{}{}
    \makeatletter
    \renewcommand{\@secnumfont}{\bfseries}
    \makeatother

    \usepackage{xcolor}
    
    \usepackage{soul}

\newtheorem{theorem}{Theorem}[section]
\newtheorem{lemma}[theorem]{Lemma}
\newtheorem{proposition}[theorem]{Proposition}
\newtheorem{corollary}[theorem]{Corollary}

\theoremstyle{definition}
\newtheorem{definition}[theorem]{Definition}
\newtheorem{example}[theorem]{Example}
\newtheorem{remark}[theorem]{Remark}

\usepackage{etoolbox}
\patchcmd{\thebibliography}{\section*}{\paragraph}{}{}

\numberwithin{equation}{section}

\newcommand{\lb}{\langle \! \langle}
\newcommand{\rb}{\rangle \! \rangle}

\def\aa{\mathfrak{a}}
\def\bb{\mathfrak{b}}

\def\lll{{\sf mt}}
\def\rrr{{\sf r}}  
\def\epi{\twoheadrightarrow}

\title[Parafree algebras and Gr\"obner-Shirshov  bases]{Parafree augmented algebras and Gr\"obner-Shirshov bases for complete augmented algebras}

\author{Sergei O. Ivanov} 
\address{
Laboratory of Modern Algebra and Applications,  St. Petersburg State University, 14th Line, 29b,
Saint Petersburg, 199178 Russia}
\email{ivanov.s.o.1986@gmail.com}

\author{Viktor Lopatkin}
\address[Viktor Lopatkin]{Laboratory of Modern Algebra and Applications,  St. Petersburg State University, 14th Line, 29b,
Saint Petersburg, 199178 Russia}
\email{wickktor@gmail.com}

\begin{document}

\maketitle

\begin{abstract}
We develop a theory of parafree augmented algebras similar to the theory of parafree groups and explore some questions related to the Parafree Conjecture. We provide an example of a  finitely generated parafree augmented algebra of infinite cohomological dimension. Motivated by this example, we prove a version of the Composition-Diamond lemma for complete augmented algebras and give a sufficient condition for augmented algebra to be residually nilpotent on the language of its relations.
\end{abstract}

\section*{Introduction} For a group $G$ we denote by $(\gamma_i(G))_{i\geq 1}$ its lower central series.  A group $G$ is called {\it parafree} if it is residually nilpotent and there is a homomorphism from a free group $F\to G,$ such that $F/\gamma_i(F) \to G/\gamma_i(G)$ is an isomorphism for any $i.$ Baumslag introduced this class of groups \cite{Ba67a}, \cite{Ba67b}, \cite{Ba68}, \cite{Ba69} searching for an example of a non-free group of cohomological dimension $1.$ When Stallings and Swan proved their famous theorem about groups of cohomological dimension one \cite{St68},  \cite{Sw69}, Baumslag realised that non-free parafree groups have cohomological dimension at least two. Despite this fact, Baumslag continued the study of homological properties which  parafree groups  share with  free groups. Baumslag raised the following conjecture, which is known as the Parafree conjecture: {\it for any finitely generated parafree group $G,$ $H_2(G)=0.$} There is also a strong version: {\it for any finitely generated parafree group $G$, $H_2(G) = 0$ and the
cohomological dimension of $G$ is at most $2.$ } For discussion of these conjectures see \cite{Co87}, \cite{CO98}. For convenience, we formulate the following conjectures related to the Parafree Conjecture: if $G$ is a parafree group, then 
\begin{itemize}
    \item[] (C1)  $H_2(G) = 0;$
    \item[] (C2)  the cohomological dimension of $G$ is at most $2.$
\end{itemize}
Bousfield showed that $H_2(\hat F)\ne 0,$ where $\hat F$ is the pro-nilpotent completion of a free group with at least two generators \cite{Bou77} (see also \cite{IM18}, \cite{IM19}). This gives an uncountable counterexample to (C1). Recently a countable but non-finitely generated counterexample to (C1) was constructed by Zaikovski, Mikhailov and the first named author  \cite{IMZ}. However, (C1) is open for finitely generated parafree groups and this is the most interesting question in this theory. The second conjecture (C2) is still open both in the general case and in the finitely generated case.

It was decided to ask the same questions for other algebraic systems with the hope that it is easier to answer them. Recenty  Zaikovski, Mikhailov and the first named author  \cite{IMZ} investigated these two conjectures for the case of Lie algebras. They constructed non-finitely generated counterexamples to analogues of both conjectures (C1), (C2) in the case of Lie algebras. However, both conjectures are still open for finitely generated parafree Lie algebras. 

In this paper we develop a theory of parafree augmented algebras and construct a finitely generated parafree augmented algebra of infinite cohomological dimension. Namely, an algebra defined by four generators and four relations 
$$A=\Bbbk\langle x_1,x_2,y_1,y_2 \mid  r_1,r_2,r_3,r_4\rangle,$$ where
$$r_1=x_1x_2+y_1^2-y_1, \hspace{1cm} r_2=x_2x_1+y_2^2-y_2, $$
$$r_3=x_1y_2-y_1x_1, \hspace{1cm}  r_4=x_2y_1-y_2x_2,$$
is a parafree augmented algebra of infinite cohomological dimension. Here $x_1,x_2$ is a set of parafree generators in $A.$ 
This gives a counterexample for (C2) in the case of augmented algebras.
By the cohomological dimension we mean the infimum of such $n$ that $H^{n+1}(A,-)=0,$ where $H^{n+1}(A,-)={\sf Ext}^{n+1}_A(\Bbbk,-).$

It is also noted in \cite{IMZ} that the counterexample for (C1), in the case of Lie algebras, gives a counterexample for (C1) in the the case of augmented algebras. Then all known results about these conjectures can be listed in the following table.
\begin{center}
\begin{tabular}{|c|c|c|c|c|}
\hline
 & \ (C1) \ & \ (C2) \   & (C1) f.g. & (C2) f.g. \\    
\hline
\text{groups}& $\times$    & ? & ? & ? \\
\hline
\text{Lie. alg.} & $\times$ & $\times$ & ? & ? \\
\hline
\text{Aug. alg.} & $\times$ & $\otimes$ & ? &  $\otimes$  \\
\hline
\end{tabular}
\end{center}
Here we denote by ``(C1)'' and ``(C2)'' the two conjectures in the general case; by ``(C1) f.g.'' and ``(C2) f.g.'' these conjectures in the finitely generated case; by ``$\times$'' the cases where counterexamples were constructed; by ``?'' the cases, where conjecture is open; and by $\otimes$ the cases of that were closed in this paper.

Note that our example is the first finitely generated example for all these types of the Parafree Conjectures. However, we do not have a finitely generated counterexample to (C1) yet even in the case of augmented  algebras. 

In  order to prove that this algebra is parafree we need to prove that the algebra is residually nilpotent. This is the most complicated part of the statement. Motivated by this example we develop a theory of Gr\"obner-Shirshov bases for complete augmented algebras in a more general setting then it was done in \cite{GH98}. This theory helps us to provide a method of proving that an augmented algebra is residually nilpotent. Namely, we give a sufficient condition for augmented algebra to be residually nilpotent on the language of its relations. Let us explain it in more detail. 

We say that an augmented algebra $A$ is of {\it finite type} if $I(A)/I(A)^2$ is a finitely generated  $\Bbbk$-module. The completion of an augmented algebra $A$ is defined as the inverse limit $\hat A=\varprojlim A/I(A)^i.$ This notion is useful for us because an augmented algebra is residually nilpotent if and only if the natural map $A\to \hat A$ is injective. An augmented algebra is called {\it complete} if the natural map $A\to \hat A$ 
is an isomorphism. Similarly, to the case of groups \cite[\S 13]{Bou77} the completion of an augmented algebra $A$ is not always complete but it is complete if $A$ is of finite type  
(\Cref{th_completion_is_complete}). If $X$ is finite, then $\Bbbk \langle X \rangle$ is of finite type  and its completion is the algebra of non-commutative power series $\Bbbk \lb X \rb.$ We prove that an algebra of finite type $A$ is complete if and only if it can be presented as a quotient of the algebra of non-commutative power series $\Bbbk \lb X \rb$ by a closed ideal $\aa\subseteq I(\Bbbk \lb X \rb)$ 
$$A\cong \Bbbk \lb X \rb/\aa,$$
where $X$ is finite. 

In order to work with the algebras of the form $\Bbbk \lb X \rb/\aa,$ it is useful to have a theory of Gr\"obner-Shirshov bases for this case which is similar to the theory for associative algebras (see \cite{BC14}, \cite{Mo94}). Namely, we need an analogue of the Composition-Diamond lemma. Such a theory was developed by Gerritzen and Holtkamp \cite{GH98}. The main difference with the theory for associative algebras is that one needs to replace the maximal term of a polynomial ${\sf MT}(p)$ by the {\it minimal} term of a power series ${\sf mt}(f)$.  However, Gerritzen and Holtkamp consider only the deg-lex order on the free monoid $W(X)$ but we need more general admissible orders. Gerritzen and Holtkamp use the notion of {\it $I$-adic basis} in their theory. We had to replace this notion to a more general notion of {\it topological basis} in order to prove the more general variant of 
the Composition-Diamond lemma (\Cref{CDA}).  \Cref{prop_I-adic} shows a connection between the notions of $I$-adic basis and topological basis.
We also want to mention a thesis of Hellstr\"om  \cite{H02} devoted to a similar theory but our approach is independent. A similar theory for commutative power series was developed by Hironaka in his famous article about resolution of singularities \cite{Hir}.

We use our variant of the Composition-Diamond lemma to prove a theorem that gives a sufficient condition for residually nilpotency of an augmented algebra on the language of its relations  (\Cref{the_main_theorem}). In statement of the theorem we use two types of Gr\"obner-Shirshov bases: the classical Gr\"obner-Shirshov basis for associative algebras and our variant for complete augmented algebras. We also use the notion of an $\mathbb N$-order: a total order such that any set bounded above is finite. Namely, we prove the following theorem.

\medskip

{\bf Theorem.} {\it Let $X$ be a finite set,  $\leq$ be an admissible order on  $W(X)$ and $\preccurlyeq$ is an admissible $\mathbb N$-order on $W(X).$ Assume that $R\subset I(\Bbbk \langle X \rangle )$ is a Gr\"obner-Shirshov basis 
 in $\Bbbk \langle X \rangle $ with respect to $\leq$ and $S\subset I(  \Bbbk \lb X \rb)$ is a Gr\"obner-Shirshov basis in $ \Bbbk \lb X \rb$ with respect to $\preccurlyeq$ such that 
 \begin{itemize}
     \item $\overline{ (\iota(R))}= \overline{(S)};$
     \item ${\sf MT}_{\leq }(R)={\sf mt}_{\preccurlyeq}(S).$
 \end{itemize}
 Then $\Bbbk \langle X \rangle /(R)$ is residually nilpotent. }

\section*{Acknowledgments}
The work on sections 1-3 was done by the first named author and it was supported by: (1) Ministry of Science and Higher Education of the Russian Federation, agreement  075-15-2019-1619; (2) the grant of the Government of the Russian Federation for the state
support of scientific research carried out under the supervision of leading scientists,
agreement 14.W03.31.0030 dated 15.02.2018; (3)  RFBR according to the research project 20-01-00030; (4) Russian Federation President Grant for Support of Young Scientists MK-681.2020.1.

The work on sections 4-6, including the main theorem, was done by the second named author and it was  supported by the Russian Science Foundation research project number 19-71-10016.

\section{Complete augmented algebras}

\subsection{Preliminaries}
Throughout the paper we denote by $\Bbbk$ a commutative ring and by $\otimes$ the tensor product over $\Bbbk$. By an algebra we always mean a $\Bbbk$-algebra.

Let $A$ be an algebra and  $\mathfrak{a}$ be an ideal of $A$. 
The {\it $\aa$-adic topology} on $A$ is the topology 
defined by taking the ideals $\mathfrak{a}^n$ as basic
neighborhoods of $0$. For an ideal $\bb \triangleleft A$ we denote by $\overline{\bb}$ its closure in this topology.  It is well-known that the closure can be computed as follows  
$\overline{\bb}= \bigcap_{n=1}^\infty   (\bb+\aa^n)$ (see \cite[Ch.III,\S 2.5]{Bour72}). In particular $\bb$ 
is closed if and only if $\bb= \bigcap_{n=1}^\infty   (\bb+\aa^n).$ The $\aa$-adic topology is 
Hausdorff if and only if the ideal $(0)$ is closed i.e.  $\cap_{n=1}^\infty \mathfrak{a}^n = (0)$. 

An {\it augmented algebra} is an algebra $A$ together with an algebra homomorphism $\varepsilon_A: A\to \Bbbk.$ This homomorphism is called augmentation and its kernel is called the augmentation ideal
$I(A)={\rm Ker}(\varepsilon_A:A\to \Bbbk).$
An augmented algebra will be always considered together with the $I(A)$-adic topology. Whenever this can be done without ambiguity we shall use the notation $I$ instead of $I(A)$.

A {\it morphism of augmented algebras} is a homomorphism $f:A\to B$ such that $\varepsilon_A=\varepsilon_B f.$ Note that for any morphism of augmented algerbas $f:A\to B$ we have $f(I(A)^n)\subseteq I(B)^n.$  In particular, $f$ is continuous.

An augmented algebra $A$ is called {\it nilpotent} if $I(A)^n=0$ for some $n,$ and {\it residually nilpotent} if $\bigcap_{n=1}^\infty I(A)^n=0.$ Note that an augmented algebra $A$ is residually nilpotent if and only if it is Hausdoff with respect to $I(A)$-adic topology.

For an augmented algebra $A$ we  define its {\it associated graded algebra} ${\sf Gr}(A)$ as follows
$${\sf Gr}(A)=\bigoplus_{n\geq 0} {\sf Gr}_n(A), \hspace{1cm}   {\sf Gr}_n(A)=I(A)^n/I(A)^{n+1}$$
with the multiplication induced by the multiplication in $A$.

\subsection{The algebra of non-commutative power series}
Let $X=\{x_1,\dots,x_N\}$ be a finite set and $W(X)$ be the free monoid generated by $X.$ Elements of $W(X)$ are called {\it words} in the alphabet $X$. 
We denote by $W_n(X), W_{<n}(X),W_{\geq n}(X)$ the sets of words of length $n,$ of length lesser than $n,$ and of length at least $n$ respectively. If $X$ is obvious, we simplify notation $W:=W(X).$

The free algebra $\Bbbk \langle X \rangle$ consists of linear combinations $\sum_{w\in W} \alpha_w w.$ Here we assume that this linear combination is finite i.e., the set $\{ w\in W\mid \alpha_w\ne 0\}$ is finite. We treat $\Bbbk\langle X \rangle$ as an augmented algebra with the augmentation $\varepsilon:  \Bbbk\langle X \rangle \to \Bbbk$ such that $ \varepsilon(x)=0$ for any $x\in X.$ 

Note that $\Bbbk\langle X \rangle$ is a free augmented algebra in the following sence. For any augmented algebra $A$  and a map $X\to I(A)$  there exists a unique extension of this map to a morphism of augmented algebras $\Bbbk \langle X \rangle \to A.$ 

It is easy to see that for any $n\geq 1$ the ideal  $I( \Bbbk\langle X \rangle)^n$  consists of all finite sums of the form $\sum_{w\in W_{\geq n}} \alpha_w w.$

Denote by $\Bbbk\langle\!\langle X \rangle\!\rangle$ the algebra of non-commutative power series. Its elements are (not necessarily finite) formal sums $\sum_{w\in W(X)} \alpha_w w$. The product in this algebra  is given by 
$$\left(\sum_{w\in W} \alpha_w w\right)\left(\sum_{w\in W} \beta_w w\right)=\sum_{w\in W} \left(\sum_{\{(u,v)\in W^2\mid uv=w\} }\alpha_u\beta_v \right) w.$$

It is well defined, because, for any $w\in W$, the number of decompositions $w=uv$ is finite.  We also treat this algebra as an augmented algebra with the augmentation   $\varepsilon(\sum \alpha_w w )=\alpha_1.$ 

For any $n\ge 1$ we consider the following morphism of augmented algebras 
$$p_n:\Bbbk \langle\!\langle X \rangle \! \rangle \longrightarrow  \Bbbk \langle X \rangle/I^n,$$
defined by the formula $\sum \alpha_ww \mapsto \sum_{w\in W_{<n}(X)} \alpha_ww.$ 

\begin{lemma}\label{lemma_power_ideal}
Let $X=\{x_1,\dots,x_N\}$ be a finite set and $n\geq 1$. Then
$I(\Bbbk \langle\!\langle X \rangle \! \rangle )^n= {\rm Ker}(p_n).$ In other words, $p_n$ induces an isomorphism
$$\Bbbk\langle\!\langle X \rangle\!\rangle/I^n \cong \Bbbk \langle X \rangle/I^n.$$
\end{lemma}
\begin{proof}
If we set $J_n={\rm Ker}(p_n),$ then we have $J_n\cdot J_m\subseteq J_{n+m}$, and $J_1=I.$ It follows that $I^n\subseteq J_n.$ Take an element $a\in J_n$, and let $a=\sum_{w\in W_{\geq n}} \alpha_ww$.  It is clear that for any $n\ge 1$, $w \in W_{\ge n}$ we have $w=x_{i_1}x_{i_2}\dots x_{i_n}w',$ where $i_1,\dots,i_n\in \{1,\dots,N\}$. Since $\{1,\dots,N \}^n$ is finite, we can present $a$ as a finite sum
$$a=\sum_{(i_1,\dots,i_n)\in \{1,\dots,N\}^n} x_{i_1}\dots x_{i_n}\cdot a_{i_1,\dots,i_n},$$
where $a_{i_1,\dots,i_n}\in \Bbbk\langle\!\langle X \rangle\!\rangle.$ Obviously $x_{i_1}\dots x_{i_n}\cdot a_{i_1,\dots,i_n}\in I^n.$ Hence $a\in I^n.$ Therefore $I^n=J_n.$
\end{proof}

\begin{remark}
If $X=\{x_1, x_2, \dots\}$ is infinite, one can also define $\Bbbk \langle\!\langle X \rangle \! \rangle $ but  \Cref{lemma_power_ideal} fails because of $x_2^2+x_3^3+ x_4^4+\dots\notin I(\Bbbk \langle\!\langle X \rangle \! \rangle )^2.$  
\end{remark}

\subsection{Completion of augmented algebras} 
Let $A$ be an augmented algebra. Its completion is defined as the inverse limit
$$ \hat A= \varprojlim A/I(A)^n.$$
The completion $\hat A$ is also an augmented algebra with the augmentation induced by the augmentation of $A.$ It is easy to see that $\Bbbk\langle\!\langle X \rangle \!\rangle$ is the completion of $ \Bbbk\langle X \rangle.$ 

\begin{lemma}\label{lemma_completion_is_res_nil}
For any augmented algebra $A$ the completion $\hat A$ is residually nilpotent. 
\end{lemma}
\begin{proof}
 Set $J_n={\sf Ker}(\hat A\to A/I(A)^n ).$ Then $\bigcap J_n=0$ and $J_n\cdot J_m \subseteq J_{n+m}$ for any $n,m.$ Since $I(\hat A)=J_1,$ we obtain $I(A)^n\subseteq J_n.$ Hence $\bigcap I(\hat A)^n=0.$ 
\end{proof}

\begin{lemma}\label{lemma_surjections}
Let $f:A\to B$ be a morphism of augmented algebras. Then the following statements are equivalent:
\begin{enumerate}
\item $\hat f:\hat A \to \hat B$ is surjective;
\item $I(A)/I(A)^2\to I(B)/I(B)^2$ is surjective;
\item ${\sf Gr}(A)\to {\sf Gr}(B)$ is surjective;
\item  $A/I(A)^n\to B/I(B)^n$ is surjective for any $n.$
\end{enumerate}
\end{lemma}
\begin{proof}
$(3)\Rightarrow (2)$ is obvious.

$(2)\Rightarrow (3).$
Note that the multiplication induces a well defined surjective map $I(A)/I(A)^2 \otimes I(A)^n/I(A)^{n+1} \to I(A)^{n+1}/I(A)^{n+2}$ for any $n\geq 1.$ Therefore, by induction on $n$, we obtain the surjectivity of  $I(A)^n/I(A)^{n+1} \to I(B)^n/I(B)^{n+1}$ for any $n\geq 1$.

Thus we have proved  $(2) \Leftrightarrow (3).$

$(3) \Rightarrow (4).$ It follows by induction. 

$(4)\Rightarrow (2).$ Since $A/I(A)^2\cong \Bbbk\oplus I(A)/I(A)^2,$ the surjectivity of  $A/I(A)^2\to B/I(B)^2$ implies the surjectivity of $I(A)/I(A)^2\to I(B)/I(B)^2.$  

Thus we have proved  $(2) \Leftrightarrow (3) \Leftrightarrow (4).$

$(1) \Rightarrow (4).$ Since $A/I(A)^n$ is a quotient of $\hat A,$ the surjectivity of $\hat A\to \hat B$ implies the surjectivity of $A/I(A)^n\to B/I(B)^n.$ 

Now we only need to prove $(2)\& (3)\&(4) \Rightarrow (1).$ Set $K_n={\rm Ker}(A/I(A)^n\to B/I(B)^n)$ and
consider the following commutative diagram with exact rows and columns
$$
\begin{tikzcd}
&  & I(A)^n/I(A)^{n+1}\arrow{r}\arrow{d} & I(B)^{n}/I(B)^{n+1}\arrow{r} \arrow{d} & 0 \\
0\arrow{r} &K_{n+1} \arrow{d} \arrow{r} & A/I(A)^{n+1} \arrow{r} \arrow{d} & B/I(B)^{n+1} \arrow{r} \arrow{d} & 0\\
0\arrow{r} & K_{n} \arrow{r} & A/I(A)^{n} \arrow{r} & B/I(B)^{n} \arrow{r} & 0
\end{tikzcd}
$$
Then by the snake lemma, $K_{n+1}\to K_n$ is surjective. Hence, by the Mittag-Leffler condition, the map $\varprojlim A/I(A)^n\to \varprojlim B/I(b)^n$ is surjective. 
\end{proof}

\subsection{Augmented algebras of finite type}

We say that an augmented algebra $A$ is of {\it finite type} if $I(A)/I(A)^2$ is a finitely generated $\Bbbk$-module.

\begin{lemma}\label{lemma_free_is_finite}
For any finite set $X$ the algebra $\Bbbk \lb X\rb$ is of finite type.
\end{lemma}
\begin{proof}
It follows from  \Cref{lemma_power_ideal}.
\end{proof}

\begin{proposition}\label{lemma_finite_type}
Let $A$ be an augmented algebra. Then the following statements  are equivalent: 
\begin{enumerate}
    \item $A$ is of finite type;
    \item ${\sf Gr}_n(A)$ is finitely generated as a $\Bbbk$-module for any $n\geq 0;$
    \item $A/I(A)^n$ is finitely generated as a $\Bbbk$-module for any $n;$
    \item $\hat A\cong \Bbbk \lb X \rb/\aa $ for some closed ideal $\aa \subseteq I(\Bbbk \lb X \rb)$ and finite $X.$
\end{enumerate}
\end{proposition}
\begin{proof}
$(2) \Rightarrow (1)$ and $(3)\Rightarrow(1)$ are obvious. 

$(4) \Rightarrow (1).$ By \Cref{lemma_free_is_finite}, $\hat A$ is of finite type. The epimorphism $\hat A \twoheadrightarrow  A/I(A)^2$ induces an epimorphism $\hat A/I(\hat A)^2 \twoheadrightarrow  A/I(A)^2.$ It follows that $A$ is of finite type.

$(1)\Rightarrow (2),(3),(4).$ Consider a finite set $X$ and a map $X\to I(A)$ such that the image of the composition $X\to I(A)\to I(A)/I(A)^2$ generates $I(A)/I(A)^2.$ 
This map induces a morphism $\Bbbk \langle X \rangle \to A$ such that $I(\Bbbk\langle X \rangle )/I(\Bbbk\langle X \rangle )^2\to I(A)/I(A)^2$ is surjective. Then $(2)$ and $(3)$ follow from  \Cref{lemma_surjections}. Moreover, it follows that the map $\Bbbk \lb X \rb\to \hat A$ is surjective. By  \Cref{lemma_completion_is_res_nil} the augmented algebra $\hat A$ is Hausdorff in the $I(A)$-adic topology and hence, $0$ is a closed ideal. Therefore $\aa={\rm Ker}(\Bbbk \lb X \rb\to \hat A )$ is closed and $(4)$ follows.
\end{proof}

\subsection{Complete augmented algebras}

An augmented algebra $A$ is called {\it complete} if the map $A\to \hat A$ is an isomorphism.

The completion $\hat A$ of an augmented algebra $A$ is {\bf not necessarily complete} as an augmented algebra, because the $I(\hat A)$-adic topology on $\hat A$ does not need to be equal to the inverse limit topology. However, we prove (\Cref{th_completion_is_complete}) that $\hat A$ is complete if $A$ is of finite type.

\begin{lemma}
For any finite set $X$ the algebra $\Bbbk \lb X\rb$ is complete.
\end{lemma}
\begin{proof}
It follows from  \Cref{lemma_power_ideal}.
\end{proof}
 
\begin{lemma}
Let $A\to B$ be a morphism of complete augmented algebras. Then the following statements are equivalent.
\begin{enumerate}
\item $A \to B$ is surjective;
\item $I(A)/I(A)^2\to I(B)/I(B)^2$ is surjective.
\item ${\sf Gr}(A)\to {\sf Gr}(B)$ is surjective;
\item  $A/I(A)^n\to B/I(B)^n$ is surjective for any $n.$
\end{enumerate}
\end{lemma}
\begin{proof}
It follows from  \Cref{lemma_surjections}.
\end{proof}

\begin{lemma}\label{lemma_complete_quo}
Suppose $A$ is a complete augmented algebra and  $\aa\subseteq I(A)$ is a closed ideal. Then $A/\aa$ is also complete. 
\end{lemma}
\begin{proof}
 \Cref{lemma_surjections} implies that the map $A\to \widehat{A/\aa} $ is surjective. Hence $A/\aa \to  \widehat{A/\aa}$ is surjective. Since $\aa$ is closed, we have $\bigcap_{n=1}^\infty (\aa+I(A)^n)=\aa.$ Then map is injective because $\bigcap_{n=1}^\infty  I(A/\aa)^n= \bigcap_{n=1}^\infty (\aa+I(A)^n)/\aa=\aa/\aa=0.$
\end{proof}

\begin{theorem}\label{th_completion_is_complete}
Let $A$ be an augmented algebra  of finite type. Then $\hat A$ is complete and the map $A\to \hat A$ induces  isomorphisms $$A/I(A)^n\cong \hat A/I(\hat A)^n, \hspace{1cm} {\sf Gr}(A)\cong {\sf Gr}(\hat A).$$
In particular, $\hat A$ is of finite type. 
\end{theorem}
\begin{proof} By \Cref{lemma_finite_type}, $\hat A\cong \Bbbk \lb X \rb/\aa,$ where $\aa\subseteq I(\Bbbk \lb X \rb)$ is a closed ideal. Then, by \Cref{lemma_complete_quo}, $\hat A$ is complete.

Now we aim to prove  $A/I(A)^n\cong \hat A/I(\hat A)^n$ and $I^n(A)/I^{n+1}(A)\cong I^n(\hat A)/I^{n+1}(\hat A).$ The second one easily follows from the first one. So we need to prove only the first one. Since $A\to \hat A$ induces an isomorphism on completions,   \Cref{lemma_surjections} implies that $A/I(A)^n\to \hat A/I(\hat A)^n$ is surjective. On the other hand it is injective because the composition $A/I(A)^n \to \hat A/I(\hat A)^n \to A/I(A)^n$ is identical. 
\end{proof}

Denote by ${\sf Aug^{fin}}$ the category of augmented algebras of finite type. Consider its full subcategory ${\sf CAug^{fin}}$ consisting of complete algebras of finite type.

\begin{proposition}
 The operation of completion defines the left adjoint functor to the functor of embedding $\iota: {\sf CAug^{fin}} \hookrightarrow {\sf Aug^{fin}}.$ 
$$\widehat{(\cdot)} : {\sf Aug^{fin}} \leftrightarrows {\sf CAug^{fin}}: \iota.$$
\end{proposition}
\begin{proof}
By \Cref{th_completion_is_complete}, the functor of completion $\widehat{(\cdot)} : {\sf Aug^{fin}} \to {\sf CAug^{fin}}$ is well defined. Then we only need to prove that for any morphism from an augmented algebra to a complete augmented algebra $f:A\to C$ there exists a unique $\hat f:\hat A\to C$ that makes the diagram is commutative. Such a morphism  can be constructed as the composition of the completion of the original morphism and the isomorphism  $\hat A\to \hat C \cong C.$ It is unique because the image of $A\to \hat A$ is dence in $\hat A$ and the morphisms $\hat A\to C$ are continuous. 
\end{proof}

\begin{theorem}\label{theorem1.2}
Let $A$ be an augmented algebra of finite type. Then $A$ is complete if and only if  there exists a finite set $X$ and a closed ideal $\aa \subseteq I( \Bbbk \langle\!\langle X \rangle \!\rangle)$ such that $$A\cong \Bbbk \langle \! \langle X \rangle\!\rangle /\aa.$$
\end{theorem}
\begin{proof}
Any algebra of the form $\Bbbk \langle \! \langle X \rangle\!\rangle /\aa$ is complete by  \Cref{lemma_complete_quo}. On the other hand, by  \Cref{lemma_finite_type}, we have $A\cong \hat A \cong \Bbbk \lb X\rb /\aa$ for a closed ideal $\aa\subseteq I(\Bbbk \lb X\rb).$
\end{proof}

\section{Gr\"obner-Shirshov bases for closed ideals of $\Bbbk \lb X \rb$}

The goal of this section is to prove a variant of the Composition-Diamond lemma for closed ideals of the algebra of non-commutative power series $\Bbbk\lb X \rb$ in a more general setting then it was done by Gerritzen and Holtkamp  \cite{GH98}. Namely we prove it with respect to an arbitrary admissible $\mathbb N$-order on the free monoid $W(X)$  and over an arbitrary commutative ring $\Bbbk$. We also want to mention a thesis of Hellstr\"om  \cite{H02} devoted to a similar question but our approach is independent.

\subsection{Topological bases and $I$-adic bases}

Gerritzen and Holtkamp in  \cite{GH98} use the notion of $I$-adic basis. In our generalisation we need a more general notion of topological basis that is introduced in this subsection.

Here  $\Bbbk$ is considered as a discrete topological commutative ring. 

\begin{definition}[topological sum] Let $M$ be a Hausdorff topological $\Bbbk$-module and 
let $(m_i)_{i\in \mathfrak{I}}$ be a family in $M$ indexed by some not necessarily finite set 
$\mathfrak{I}.$ An element $m$ is called {\it topological sum} of this family if for any neighbourhood $U$ of $m$ there exists a finite subset $\mathfrak{F}\subseteq \mathfrak{I}$ such that for any  finite $\mathfrak{F}'$,
satisfying 
$\mathfrak{F}\subseteq  \mathfrak{F}' \subseteq \mathfrak{I}$, 
we have $\sum_{i\in \mathfrak{F}'} m_i\in U.$ If a topological sum exists, then the family $(m_i)_{i\in\mathfrak{I}}$ is called {\it summable}.  Since $M$ is Hausdorff, we see that a topological sum, if it exists, is unique and we denote it by
$$m=\sum_{i\in \mathfrak{I}} m_i.$$
We do not call topological sums ``series'' so as not to confuse with elements of $\Bbbk \lb X\rb$.
\end{definition}

\begin{definition}[topological basis] Let $M$ be a Hausdorff topological $\Bbbk$-module and let $(b_i)_{i\in \mathfrak{I}}$ be a family of elements of $M$. A {\it topological linear combination} of $(b_i)_{i\in \mathfrak{I}}$ is a topological sum of the form $\sum_{i\in \mathfrak{I}} \alpha_i b_i,$ where $(\alpha_i)_{i\in \mathfrak{I}} $ is a family of scalars such that $(\alpha_ib_i)_{i\in \mathfrak{I}}$ is summable. 
The family $(b_i)_{i\in \mathfrak{I}}$ is called {\it topological basis} of $M$ if any element of $M$ can be uniquely presented as its topological linear combination.  
\end{definition}

\begin{lemma}\label{lemma_top_sum}
Let $A$ be a complete augmented algebra, let $(a_i)_{i\in \mathfrak{I}},a_i\in A$ be a family and $\mathfrak{I}_n=\{i\in \mathfrak{I}\mid a_i\notin I(A)^n\}.$ Then $(a_i)_{i\in \mathfrak{I}}$ is summable if and only if $\mathfrak{I}_n$ is finite for any $n$. \end{lemma}
\begin{proof}
Assume that $(a_i)_{i\in \mathfrak{I}}$ is summable and $s=\sum_{i\in \mathfrak{I}} a_i$. For a finite $\mathfrak{F}\subseteq \mathfrak{I}$ we consider  a sum  $s_\mathfrak{F}=\sum_{i\in \mathfrak{F}} a_i.$  Then for any $n$ there exists a finite $\mathfrak{F}_n\subseteq \mathfrak{I}$ such that for any finite $\mathfrak{F}'$ satisfying  $\mathfrak{F}_n \subseteq \mathfrak{F}' \subseteq \mathfrak{I}$ we have $s-s_{\mathfrak{F}'} \in I(A)^n.$ Then for any $i\notin \mathfrak{F}_n$ we have $a_i=(s-s_{\mathfrak{F}_n}) - (s-s_{\mathfrak{F_n}\cup\{i\}}) \in I(A)^n.$  
It follows that $\mathfrak{I}_n\subseteq \mathfrak{F}_n$ and hence $\mathfrak{I}_n$ is finite. 

Now  assume that $\mathfrak{I}_n$ is finite and set $s_n=\sum_{i\in \mathfrak{I}_n} a_i.$ Then $s_n-s_{n+1}\in I^n$ and hence $s_n$ converges to some element $s$ so that $s-s_n\in I^n.$ For any neighbourhood $U$ of $s$ there is $n$ such that $s+I^n\subseteq U.$ Since $a_i\in I^n$ for any $i\notin \mathfrak{I}_n,$ we obtain that $\sum_{i\in \mathfrak{F}'} a_i\in s+I^n\subseteq U$ for any finite $\mathfrak{F}'$ such that $ \mathfrak{I}_n \subseteq  \mathfrak{F}' \subseteq \mathfrak{I}.$ Then $(a_i)_{i\in \mathfrak{I}}$ is summable. 
\end{proof}

\begin{definition}[$I$-adic basis] 
Let $A$ be a complete augmented algebra, $(b_i)_{i\in\mathfrak{I}},b_i\in A$ be a family and $\mathfrak{I}_n=\{i\in \mathfrak{I}\mid b_i\notin I(A)^n\}.$ Following \cite{GH98} we say that $(b_i)_{i\in \mathfrak{I}}$ is an {\it $I$-adic basis} if $(p_n(b_i))_{i\in \mathfrak{I}_n}$ is a basis of the $\Bbbk$-module $A/I(A)^n$ for any $n\geq 0.$
\end{definition}

Here we give an interpretation of the notion of the $I$-adic basis in terms of topological bases.

\begin{proposition}\label{prop_I-adic}
Let $A$ be a complete augmented algebra, $(b_i)_{i\in\mathfrak{I}},b_i\in A$ be a family and $\mathfrak{I}_n=\{i\in \mathfrak{I}\mid b_i\notin I(A)^n\}.$ Then $(b_i)_{i\in \mathfrak{I}}$ is an $I$-adic basis if and only if $(b_i)_{i\in \mathfrak{I}\setminus \mathfrak{I}_n}$ is a topological basis of $I(A)^n$ for any $n.$
\end{proposition}
\begin{proof} 
Assume that $(b_i)_{i\in \mathfrak{I}}$ is an $I$-adic basis. Fix $n\geq 0.$ Since $(p_n(b_i))_{i\in \mathfrak{I}_n}$ is a basis for any $n,$ we obtain that $(p_m(b_i))_{i\in \mathfrak{I}_m\setminus \mathfrak{I}_n}$ is a basis of $I^n/I^m$ for any $m>n.$ Take an element $a\in I^n.$ Consider a linear combitation $p_m(a)=\sum_{i\in \mathfrak{I}_m \setminus \mathfrak{I}_n} \alpha_{i,m} p_m(b_i).$ 
The fact that the map $p_{m}:A\to A/I^m$ factors as $A\to A/I^{m+1} \to A/I^m,$ implies that $\alpha_{i,m}=\alpha_{i,m+1}$ for any $i\in \mathfrak{I}_m\setminus \mathfrak{I}_n.$ Hence we can define $\alpha_i:=\alpha_{i,m}$ and obtain $a=\sum_{i\in \mathfrak{I}\setminus \mathfrak{I}_n} \alpha_{i} b_i.$ 

Prove that this presentation is unique. 
It is enough to prove that the equality  $\sum_{i\in \mathfrak{I} \setminus \mathfrak{I}_n } \alpha_i b_i=0$ 
implies $\alpha_i=0$ for any $i\in \mathfrak{I}\setminus \mathfrak{I}_n.$ 
If  $\sum_{i\in \mathfrak{I}\setminus \mathfrak{I}_n} \alpha_i b_i=0,$ then  
$\sum_{i\in \mathfrak{I}_n \setminus  \mathfrak{I}_m} \alpha_i p_m(b_i)=0$ for any $m>n,$ where this sum is a usual linear combination. Hence $\alpha_i=0$ for any $i\in\mathfrak{I}_n \setminus   \mathfrak{I}_m$ and for any $m> n.$

Now assume that  $(b_i)_{i\in \mathfrak{I}\setminus \mathfrak{I}_n}$ is a topological basis of $I^n$ for any $n$ and prove that it is an $I$-adic basis.
First we prove that $(p_n(b_i))_{i\in \mathfrak{I}_n}$ spans $A/I^n.$ Take an element $a\in A/I^n$ and any its preimage $a'\in A.$ Present $a'$ as a topological linear combination  $a'=\sum_{i\in \mathfrak{I}} \alpha_ib_i.$ Since $(\alpha_ib_i)_{i\in \mathfrak{I}}$ is summable, then $\mathfrak{F}_n:=\{i\in \mathfrak{I}\mid p_n(\alpha_ib_i)\ne 0 \}$ is finite. Note that $\mathfrak{F}_n\subseteq \mathfrak{I}_n.$ Then $a=\sum_{i\in \mathfrak{F}_n} \alpha_i b_i$ and hence $(p_n(b_i))_{i\in \mathfrak{I}_n}$ spans $ A/I^n.$ 
 
 Prove that $(p_n(b_i))_{i\in \mathfrak{I}_n}$ is linearly independent. Assume $\sum_{i\in \mathfrak{I}_n} \alpha_i p_n(b_i)=0.$ Then $\sum_{i\in \mathfrak{I}_n} \alpha_i b_i\in I^n.$ Therefore $\sum_{i\in \mathfrak{I}_n} \alpha_i b_i= 
 \sum_ {i\in  \mathfrak{I} \setminus \mathfrak{I}_n} \beta_i b_i$ for some family $(\beta_i)_{i\in \mathfrak{I} \setminus \mathfrak{I}_n}.$ Hence  $\sum_{i\in \mathfrak{I}} \gamma_i b_i=0,$ where $\gamma_i=\alpha_i$ for $i\in \mathfrak{I}_n$ and $\gamma_i=-\beta_i$ for $i\in \mathfrak{I}\setminus \mathfrak{I}_n.$ It follows that $\gamma_i=0$ for any $i\in \mathfrak{I}$ and hence $\alpha_i=0$ for any $i\in \mathfrak{I}_n.$ Thus $(p_n(b_i))_{i\in \mathfrak{I}_n}$ is linearly independent.
\end{proof}

\subsection{Composition-Diamond lemma}

An {\it $\mathbb N$-order} is an order $\leq$ on some countable set which is isomorphic to $\mathbb N.$ In other words, $\mathbb N$-order is a total order such that any bounded above subset is finite.

An  order $<$ on a monoid $M$ is called  \textit{admissible} if it is total and satisfies $1< m$ for all $m\in M\setminus \{1\}$, and $m_1 <  m_2$ implies $m'm_1m''  <  m'm_2m''$ for any $m',m''\in M$. Note that admissible orders can exist only on cancellative monoids.

Further we assume that  $X=\{x_1,\dots,x_N\}$ is a finite set, where $x_1,\dots,x_N$ are distinct elements, and we assume that $W:=W(X)$ is endowed by an admissible $\mathbb N$-order such that $x_1<x_2<\dots <x_N$. 
For example, deg-lex order is an admissible $\mathbb N$-order on $W$.

For any $f=\sum f_w w\in \Bbbk\lb X\rb,$ where $f_w\in \Bbbk,$ we set $\mathrm{supp}(f):=\{w \in W: f_w \ne 0\}.$ If $f\ne 0,$ we denote by $\lll(f)$ the minimal element of $\mathrm{supp}(f).$ We call $\lll(f)$ \textit{the minimal monomial} of $f.$ For any subset $S\subseteq \Bbbk \lb X \rb \setminus 0 $ we set $\lll(S)=\{\lll(s)\mid s\in S\}.$ Note that if $\aa$ is a closed ideal of $\Bbbk \lb X \rb,$ then $\lll(\aa\setminus 0 )$ is a monoid ideal of the monoid $W(X).$ Further, we set $\rrr(f)=f- f_{\lll(f)}\lll(f).$ Then
$$f=f_{\lll(f)} \lll(f)+\rrr(f)$$
and $\lll(\rrr(f))>\lll(f).$ A power series $f$ is called {\it monic}, if $f_{\lll(f)}=1.$ A set $S$ of power series is called monic if it consists of monic power series. 

For a subset $S\subseteq \Bbbk \lb X \rb$ we denote by $\overline{(S)}$ the closed ideal of $S\subseteq \Bbbk \lb X \rb$  generated by $S.$

\begin{definition}[Gr\"obner--Shirshov basis]
 A monic set  $S\subseteq \Bbbk \lb X \rb $ is called a \emph{Gr\"obner--Shirshov basis} if  $\lll(\overline{(S)}\setminus 0)=W\lll(S)W.$ In other words, the monoid ideal $\lll(\overline{(S)}\setminus 0)$ is generated by $\lll(S).$
\end{definition}

\begin{definition}[Compositions]
 Let $f,g \in \Bbbk \lb X \rb$ be two monic power series and $w\in W$. We introduce two types of elements that will be called {\it compositions} of $f$ and $g$ with respect to $w:$
\begin{itemize}
\item If $w = \lll(f)\cdot  u = v \cdot \lll(g)$ for some $u,v\in W$, $u,v \ne 1$, then the power series $f u -  vg$ is called the {\it intersection composition} of $f$ and $g$ with respect to $w$.
\item If $w = \lll(f) = u\cdot \lll(g)\cdot v$ for some $u,v \in W$, then the power series $f -u g v$ is called the {\it inclusion composition} of $f$ and $g$ with respect to $w$.
\end{itemize}
\end{definition}

Note that for any composition $h$ of $f$ and $g$ with respect to $w$ we have $\lll(h)> w$ and $h\in \overline{(f,g)}.$  

\begin{definition}[Trivial modulo $(S,w)$]
Let $S \subseteq \Bbbk\lb X \rb$ be a monic set and $w\in W.$ We say that $f \in \Bbbk\lb X\rb$ is  {\it trivial modulo} $(S,w)$,
if $f$ can be presented as an topological sum
$$f = \sum_{i\in \mathbb N}  g_is_ih_i, \hspace{0.5cm} \text{such that} \hspace{0.5cm} \lll(g_i) \lll(s_i) \lll(h_i) > w$$
 $s_i\in S,g_i,h_i\in \Bbbk \lb X \rb$  for any $i\in \mathbb N.$ 
 
We say that $f_1$ and $f_2$ are equal modulo $(S,w),$ denoted by $$f_1\equiv f_2 \mod (S,w),$$ if $f_1-f_2$ is trivial modulo $(S,w).$
\end{definition}

\begin{theorem}[{Composition--Diamond lemma}]\label{CDA} Let $X$ be a finite set, let $W=W(X)$  be the free monoid endowed by an admissible $\mathbb N$-order and $S$ be a monic subset of  $\Bbbk \lb X \rb.$ Set  $A:=\Bbbk \lb X\rb/\overline{(S)},$  $\beta(S):=W\setminus W\lll(S)W$  and denote by  $\varphi: \Bbbk \lb X\rb\to A$  the projection. Then the following statements are equivalent:
\begin{itemize}
\item[\rm (1)] $S$ is a Gr\"obner-Shirshov basis;
\item[\rm (2)]   every composition of series of $S$ with respect to a word $w$ is trivial modulo $(S,w)$;

\item[\rm (3)] If $f\in \overline{(S)}\setminus 0$ and $w<\lll(f),$ then $f$ is trivial modulo $(S,w);$

\item[\rm (4)]  $(\varphi(w))_{w\in \beta(S)}$ is a topological basis of $A.$
\end{itemize}
\end{theorem}

To prove the CD-lemma we need the following two lemmas.

\begin{lemma}\label{lemma1}
Let $S$ satisfy $(2)$ of the theorem. If
$w=u_1 \lll(s_1)  v_1=u_2  \lll(s_2)  v_2$, where
$u_1,v_1,u_2,v_2\in W,\ s_1,s_2\in S$, then $u_1s_1v_1 \equiv u_2 s_2 v_2 \mod (S,w). $
\end{lemma}
\begin{proof}
There are three cases to consider.

Case 1. Suppose that subwords $\lll(s_1)$ and $\lll(s_2)$ of $w$ are disjoint, say, $|u_2|\geq |u_1|+|\lll(s_1)|$. Then, $u_2=u_1\lll(s_1) u$ and $v_1=u \lll(s_2) v_2 $ for some $u\in W$, and so, $w=u_1\lll(s_1) u \lll(s_2) v_2.$

Now,
\begin{eqnarray*}
u_1s_1v_1- u_2s_2v_2 &=& u_1 s_1 u \lll(s_2)
v_2-u_1\lll(s_1) u  s_2 v_2\\
&=& - u_1 s_1 u (s_2-\lll(s_2)) v_2 + u_1(s_1-\lll(s_1)) u  s_2 v_2\\
&=& -u_1s_1u\rrr(s_2)v_2+u_1\rrr(s_1) us_2v_2.
\end{eqnarray*}
Since $\lll(\rrr(s_2))> \lll(s_2)$ and
$\lll(\rrr(s_1))> \lll(s_1)$, 
we conclude that 
$$\lll(u_1)\lll(s_1)\lll(u\rrr(s_2)v_2)>w, \hspace{1cm} \lll(u_1\rrr(s_1) u) \lll(s_2)\lll(v_2)>w.$$   Hence $u_1s_1v_1- u_2s_2v_2 \equiv 0 \mod (S,w).$

Case 2. Suppose that the subword $\lll(s_1)$ of $w$ contains $\lll(s_2)$ as a subword. Then $\lll(s_1)=u\lll(s_2)v,$ $u_2=u_1u$ and $v_2=vv_1,$ that is, $w=u_1u \lll(s_2)vv_1$.

We have
$
u_1 s_1 v_1-u_2 s_2 v_2 = u_1( s_1 - u s_2 v)v_1.$
Since the inclusion composition $s_1-us_2v$ is trivial modulo $(S,\lll(s_1)),$ the statement follows.

Case 3. $\lll(s_1)$ and $\lll(s_2)$ have a nonempty intersection as
a subword of $w$. We may assume that $u_2=u_1u,$ $v_1=vv_2,$ $w=u_1\lll(s_1) v v_2= u_1 u \lll(s_2) v_2$, where $u,v\ne 1.$ Then $s_1v-us_2$ is an intersection composition and, similar to the Case 2, we have $u_1s_1v_1 \equiv u_2s_2v_2 \mod
(S,w).$
\end{proof}

\begin{lemma}\label{lemma_replasement_fg}
Let $S\subseteq \Bbbk \lb X \rb $ be a monic set,  $f\in \Bbbk \lb X \rb$ and $w\in W$ such that $w<\lll(f).$ Then there exists $g\in \Bbbk \lb X \rb$ such that $f\equiv g\mod  (S,w)$ and $ {\rm supp}(g)\subseteq \beta(S).$ Moreover, if $\lll(f)\in \beta(S),$ then $\lll(g)=\lll(f).$ 
\end{lemma}
\begin{proof} In this prove for a power series $h$ and a word $w$ we denote by $h_w$ the coefficient of $h$ next to $w$ so that $h=\sum h_ww.$  Denote by $W_{>w}$ the set of words greater than $w.$ Note that ${\rm supp}(f)\subseteq W_{>w}.$
Since $W$ is $\mathbb N$-ordered, we can present $W\lll(S)W \cap W_{>w}$ as a increasing sequence of words  $W\lll(S)W\cap W_{>w} =\{w_1<w_2<w_3<\dots\}.$  

For each $n\geq 1$ we take the element 
$w_n\in W\lll(S)W \cap W_{>w}$ and decompose it as  $w_n=u_n\lll(s_n)v_n$ for some $u_n,v_n\in W, s_n\in S.$   We will construct inductively two sequences $g^{(1)},g^{(2)},\dots\in \Bbbk \lb X\rb$ and $\alpha_1,\alpha_2,\dots \in \Bbbk$  satisfying: 
\begin{itemize}
\item $g^{(n)}=f+\sum_{i=1}^n \alpha_iv_is_iu_i;$
\item ${\rm supp}(g^{(n)}) \cap \{w_1,\dots,w_{n}\}=\emptyset;$ 
\item if $\lll(f)\in \beta(S),$ then $\lll(f)=\lll(g^{(n)})$,
\end{itemize}  
for each $n\geq 1.$
Take  $\alpha_1:=-f_{w_1}$ and $g^{(1)}=f+\alpha_1u_1s_1v_1.$ Hence $w_1\notin {\rm supp}(g^{(1)})$ and, if $\lll(f)\in \beta(S),$ then $\lll(f)=\lll(g^{(1)}).$ Assume that we have already constructed $g^{(i)},\alpha_i$ for $i<n.$ Take $\alpha_n=-g^{(n-1)}_{w_n}$ and $g^{(n)}=g^{(n-1)}+\alpha_nv_ns_nu_n .$ It is easy to check that  $g^{(n)}$ satisfies the conditions. 

Since there is only finite number of words $w$ of length no more than $m,$ we obtain that there is $n_0$ such that  $\alpha_n u_n s_n v_n\in I^m$ for any $ n\geq n_0.$ Then the sequence $g^{(n)}$ converges to some power series $g=f+\sum_{i=1}^\infty \alpha_iu_is_iv_i.$ It is easy to see that $g$ satisfies all required conditions. 
\end{proof}

\begin{proof}[Proof of \Cref{CDA}]
$(2) \Rightarrow (1).$ Take $f \in \overline{(S)} \setminus 0$ and show 
that ${\sf mt}(f) \in W{\sf mt}(S)W.$ Chose $k$ big enough so that $k>|w|$ for any $w\leq {\sf mt}(f).$ For any power series $h$ we denote by  $\tilde h:=p_k(h)$ its image in $\Bbbk \lb X \rb.$ Then $\tilde f\in (  \tilde S).$ Using that 
$\Bbbk \lb X \rb/I^k \cong \Bbbk \langle X \rangle /I^k$ (\Cref{lemma_power_ideal}), 
we obtain 
\begin{equation}\label{eq_pres_f}
\tilde f=\sum_{i=1}^n\alpha_i u_i \tilde s_i v_i,
\end{equation}
for some  $u_i,v_i\in W,s_i\in S, \alpha_i \in \Bbbk\setminus\{0\}$ such that 
$$u_1 {\sf mt}(s_1) v_1 \leq u_2{\sf mt}(s_2) v_2\leq \dots \leq u_n{\sf mt}(s_n) v_n.$$ 
Set $w_i:=u_i{\sf mt}(s_i)v_i\in W{\sf mt}(S)W.$
Since $k$ is big enough, we have that ${\sf mt}(f)={\sf mt}( \sum_{i=1}^n\alpha_i u_i s_i v_i )$ and hence $w_1 \leq {\sf mt}(f).$

There are many  presentations of $\tilde f$ of the form \eqref{eq_pres_f} and now we chose the one with the maximal possible $w_1.$ Such a  presentation exists because $w_1 \leq {\sf mt}(f).$

Let 
\[
 w_1 = w_2 = \cdots = w_\ell < w_{\ell+1} \le \cdots.
\]
If $\ell =1$, then ${\sf mt}(f) = {\sf mt}( u_1 s_1 v_1 ) = w_1\in W{\sf mt}(S)W$ and the statement follows.
Assume that $\ell \ge 2$. Then by  \Cref{lemma1}, 
$$u_is_iv_i\equiv  u_1 s_1 v_1 \mod (S,w_1),$$ for $1 \le i \le \ell$. 
It follows that for $1\leq i \leq \ell$ we have 
$$u_i \tilde s_i v_i=u_1  \tilde  s_1 v_1\ +\ \sum_{j=1}^{m_i} \alpha'_{i,j}  u'_{i,j} \tilde  s'_{i,j} v'_{i,j}  $$
such that $u'_{i,j}{\sf mt}(s'_{i,j}) v'_{i,j}>w_1.$  
Therefore we get

\begin{equation}\label{eq_pres2}
     \tilde f = (\alpha_1 + \cdots + \alpha_\ell) u_1 \tilde s_1 v_1 + \sum_{j=1}^{\ell}\sum_{j=1}^{m_i} \alpha_i \alpha'_{i,j}  u'_{i,j} \tilde s'_{i,j} v'_{i,j} + \sum_{i=\ell+1}^n \alpha_i u_i \tilde s_iv_i. 
\end{equation}

We claim that $\alpha_1 + \cdots + \alpha_\ell \ne 0.$ Because if  $\alpha_1 + \cdots + \alpha_\ell = 0,$ we obtain a new presentation \eqref{eq_pres2} of the form \eqref{eq_pres_f} with a smaller $w_1,$ which contradicts to our choice of presentation \eqref{eq_pres_f}. Then $\alpha_1 + \cdots + \alpha_\ell \ne 0.$ It follows that ${\sf mt}(f)={\sf mt}(u_1s_1v_1)=w_1\in W{\sf mt}(S)W.$

$(3)\Rightarrow (2).$ Obvious. 

$(4) \Rightarrow (3).$ Take $f\in \overline{(S)}\setminus 0$ and $w\in W$ such that $w<\lll(f).$ By Lemma \ref{lemma_replasement_fg}, there exists $g$ such that $f\equiv g\mod (S,w)$ and ${\rm supp}(g)\subseteq \beta(S).$ Since $g\in \overline{(S)}$ and $(\varphi(w))_{w\in \beta(S)}$ is an topological basis of $\Bbbk\lb X \rb,$ we obtain $g=0.$

$(1) \Rightarrow (4).$ Set $A= \Bbbk \lb X \rb/ \overline{(S)}$ and prove that $(\varphi(w))_{w\in \beta(S)}$ is a topological basis of $A.$ By \Cref{lemma_replasement_fg}, we obtain that any element of $A$ can be presented as a topological linear combination of the family $(\varphi(w))_{w\in \beta(S)}.$ Prove that it is unique. It is enough to prove that if $\sum_{w\in \beta(W)} \alpha_w \varphi(w)=0,$ then $\alpha_w=0$ for any $w\in \beta(S).$ Note that by \Cref{lemma_top_sum}, any family of the form $(\alpha_ww)_{w\in \beta(S)}$ is summable in $\Bbbk \lb X \rb$. If $\sum_{w\in \beta(W)} \alpha_w \varphi(w)=0,$ then $f:=\sum_{w\in \beta(W)} \alpha_w w\in \overline{(S)}.$ Since $S$ is a Gr\"obner-Shirshov basis, we see that $\lll(f)\in W\lll(s)W.$ However ${\rm supp}(f)\subseteq \beta(S).$ This is a contradiction. 
\end{proof}

\subsection{$\theta$-lex orders on $W(X)$}

In this subsection we present a way to construct admissible orders on $W(X),$ which generalise the deg-lex order.

We denote by $<_{\sf lex}$ the lexicographical order on $W(X).$
Note that $<_{\sf lex}$ is not admissible because $1<_{\sf lex}x_1$ but $x_2 >_{\sf lex} x_1x_2.$ However, the lexicographical order is not far away from being admissible because it satisfies the following useful property. If $w_1\ne 1,$ then $w_1<_{\sf lex} w_2$ implies $uw_1 v<_{\sf lex} u w_2 v$ for any $u,v\in W(X).$ 

\begin{definition}[$\theta$-lex order] Let $P$ be a poset and let $\theta:W(X)\to P$ be a map. Then the $\theta$-lex order $<_\theta$ on $W(X)$ is defined as follows.  The inequality  $u<_\theta v$ holds if either $\theta(u)<\theta(v);$ or $\theta(u)=\theta(v)$ and $u<_{\sf lex} v.$ It is easy to check that $<_\theta$ is an order. 
\end{definition}

For example, the classical deg-lex order is the $\theta$-lex order, where $\theta={\sf deg}:W(X) \to \mathbb N$ is the degree function. The lexicographical order is the  $\theta$-lex order, where $\theta:W(X)\to \{*\}$ is the map to a one-element set. 

\begin{lemma} Let $M$ be a monoid with an admissible order and let $\theta:W(X)\to M$ be a homomorphism such that $w\ne 1$ implies $\theta(w)\ne 1.$ Then the  $\theta$-lex order is admissible. 
Moreover, if $M=\mathbb N,$ then the $\theta$-lex order is an $\mathbb N$-order.
\end{lemma}
\begin{proof} It is easy to see that the order is total.
For any $w\ne 1$ we have  $1<\theta(w),$ and hence $1<_\theta w.$ Let $w_1<_\theta w_2$ and $u,v\in W(X)$. Consider two cases:  $\theta(w_1)<\theta(w_1)$ and  $\theta(w_1)=\theta(w_2).$ If $\theta(w_1)<\theta(w_1),$ then $\theta(u)\theta(w_1)\theta(v)<\theta(u)\theta(w_2)\theta(v)$ and hence $uw_1v<_\theta uw_2v.$ If $\theta(w_1)=\theta(w_2),$ then $w_1<_{\sf lex}w_2.$ Note that $w_1\ne 1$ because $w_1=1$ would imply  $1=\theta(w_1)=\theta(w_2)$ and $w_2=1.$ Hence $w_1,w_2\in W(X) \setminus \{1\}.$ Then $uw_1v<_{\sf lex}uw_2v$ and $\theta(uw_1v)=\theta(uw_2v).$ It follows that $uw_1v<_\theta uw_2v.$

Assume that $M=\mathbb N.$ Since $\theta(x_i)>0$ for any $x_i\in X$ and $X$ is finite, for any $n\in \mathbb N$ the set of all words $w$ with $\theta(w)\leq n$ is finite. Then any bounded above set in $W(X)$ is finite. 
\end{proof}

\subsection{Counterexamples to possible generalisations of  CD-lemma}

Here we show that it is impossible to replace an $\mathbb N$-order by a well-order in \Cref{CDA} and it is impossible to replace topological bases by $I$-adic bases.

\begin{example}
Let $X=\{x_1,x_2,x_3,x_4\}.$  Consider a homomorphism $f:W(X)\to \mathbb N^4$ given by $f(x_i)=(0,\dots, 0,1,0,\dots,0),$ where $1$ is placed in $i$th position. We consider $\mathbb N^4$ as an ordered monoid with the lexicographical order. Then $W(X)$ is considered with the $\theta$-lex order.  It is easy to check that this order is an admissible well-order on $W(X).$ Consider  relations $s_n=x_2^nx_3^nx_4^n - x_1$ and a set
$S=\{s_n \mid n\geq 1 \}.$ Then $\lll(s_n)=x_2^nx_3^nx_4^n.$ Note that $S$ has no compositions and hence satisfy $(2)$ of \Cref{CDA}. Since $x_1=s_n+x_2^nx_3^nx_4^n,$ we obtain $x_1\in (S)+I^{3n},$ and hence $x_1\in \overline{(S)}.$ On the other hand $x_1\in \beta(S).$ Then $(\varphi(w))_{w\in \beta(S)}$ is not a topological basis of $\Bbbk\lb X \rb/\overline{(S)},$  then \Cref{CDA} fails because of (4). 
Therefore \Cref{CDA} has no a straight generalisation to the case of admissible well-orders.
\end{example}

\begin{example}
Let $X=\{x_1,x_2,x_3,x_4\}.$  Consider $\theta:W(X)\to \mathbb N$ such that $\theta(x_1)=\theta(x_2)=1$ and $\theta(x_3)=\theta(x_4)=3.$ Take $s=x_1x_2+x_3+x_4$ and $S=\{s\}.$ Then $\lll(s)=x_1x_2$ and $S$ is a Gr\"obner-Shirshov basis in $\Bbbk\lb X\rb$. Then $(\varphi(w))_{w\in \beta(S)}$ is a topological basis. However, we claim that it is not an $I$-basis. Indeed, $x_3,x_4\in \beta(S)$ but their images in $A/I^2$ are linearly dependent.
\end{example}

\section{Residually nilpotent augmeted algebras}

An augmented algebra $A$ is called {\it residually nilpotent} if $\bigcap_{n\geq 0} I(A)^n=0.$ In other words, $A$ is residually nilpotent if and only if the map to the completion $A\to \hat A$ is injective. This section is devoted to a theorem, that is based on the CD-lemma, gives a method to find out whether an augmented algebra is residually nilpotent.

In this section we will strictly distinguish between non-commutative polynomials  $p\in \Bbbk \langle X \rangle$ and  corresponding power series which will be denoted by $\iota(p)\in \Bbbk\lb X \rb.$ Polynomials are denoted by $p,q\dots \in\Bbbk \langle x\rangle$ and power series are denoted by $f,g,\dots \in \Bbbk \lb  X\rb.$
So we have an embedding
$$\iota: \Bbbk \langle X \rangle \longrightarrow \Bbbk \lb X \rb.$$

\begin{lemma}\label{lemma_presentation_completion}
Let $X$ be a finite set,  $R\subset  I( \Bbbk \langle X \rangle),$ and $A= \Bbbk \langle X \rangle/(R).$ Then $$\hat A\cong \Bbbk \lb X\rb /\overline{(\iota(R))}.$$
\end{lemma}
\begin{proof}
Since $\iota$ induces an isomorphism $\Bbbk\langle X \rangle/I^n\cong  \Bbbk\lb X \rb/I^n$ (\Cref{lemma_power_ideal}), we
have $A/I^n\cong \Bbbk\lb X \rb/((\iota (R))+I^n).$  Note that $(\iota (R))+I^n=\overline{(\iota (R))}+I^n.$
Using that $\Bbbk \lb X \rb/\overline{(\iota(R))}$ is 
complete 
(\Cref{lemma_complete_quo}), we obtain 
$\hat A\cong \varprojlim \Bbbk\lb X \rb/(\overline{(\iota (R))}+I^n) \cong  \Bbbk \lb X \rb/\overline{(\iota(R))}.$
\end{proof}

We will use both the classical theory of Gr\"obner-Shirshov bases for non-commutative polynomials $\Bbbk\langle X \rangle$ and our theory of Gr\"obner-Shirshov bases for $\Bbbk \lb X \rb .$ Let us remind the definition for non-commutative polynomials and introduce some notation.

Assume that the free monoid $W(X)$ is endowed by some admissible order $\leq$. Then for a polynomial $p\in \Bbbk \langle X \rangle$ we denote by ${\sf MT}_{\leq}(p)$ the {\it maximal term} of $p.$
$${\sf MT}_{\leq}(p)={\rm max}({\rm supp}(p)), \hspace{1cm} p\in \Bbbk\langle X \rangle.$$ A polynomial is $p$ called {\it monic} if the coefficient of ${\sf MT}_{\leq}(f)$ in $f$ is $1.$ We say that a subset $R \subset \Bbbk \langle X \rangle$ is
a Gr\"obner-Shirshov basis with respect to $\leq $ if it consists of monic polynomials and $W(X) {\sf MT}_{\leq}(R)W(X)={\sf MT}_{\leq}((R)),$ where $(R)$ is the ideal generated by $R.$

Now assume that $\preccurlyeq$ is an admissible $\mathbb N$-order on $W(X).$
For a power series $f$, as before, we denote by  ${\sf mt}_{\preccurlyeq}(g)$ the {\it least} monomial of $f$ 
$${\sf mt}_{\preccurlyeq}(f)=\min ({\rm supp}(f)), \hspace{1cm} f\in \Bbbk \lb X \rb.$$
As before, a power series $f$ is called monic if the coefficient of ${\sf mt}_{\preccurlyeq}(f)$ in $f$ is $1;$ and we say that $S\subset \Bbbk \lb X \rb$ is Gr\"obner-Shirshov basis with respect to $\preccurlyeq$, if $W(X){\sf mt}_{\preccurlyeq}(S)W(X)={\sf mt}(\overline{(S)}),$ where $\overline{(S)}$ is a closed ideal generated by $S.$

\begin{theorem}\label{the_main_theorem} Let $X$ be a finite set,  $\leq$ be an admissible order on $W(X)$ and $\preccurlyeq$ is an admissible $\mathbb N$-order on $W(X).$ Assume that $R\subset I(\Bbbk \langle X \rangle )$ is a Gr\"obner-Shirshov basis 
 in $\Bbbk \langle X \rangle $ with respect to $\leq$ and $S\subset I(  \Bbbk \lb X \rb)$ is a Gr\"obner-Shirshov basis in $ \Bbbk \lb X \rb$ with respect to $\preccurlyeq$ such that 
 \begin{itemize}
     \item $\overline{(\iota(R) )}=\overline{(S)};$
     \item ${\sf MT}_{\leq }(R)={\sf mt}_{\preccurlyeq}(S).$
 \end{itemize}
 Then $\Bbbk \langle X \rangle /(R)$ is residually nilpotent. 
\end{theorem}

\begin{proof} Set $A=\Bbbk \langle X \rangle/(R).$ Then, by \Cref{lemma_presentation_completion}, we have $\hat A\cong \Bbbk \lb X \rb/\overline{(\iota(R))}=\Bbbk \lb X \rb/\overline{(S)} .$ 
Consider the projections $\psi:\Bbbk \langle X \rangle \to A$ and $\varphi: \Bbbk \lb X \rb \to \hat A.$ By the classical Composition--Diamond for $\Bbbk \langle X \rangle$, the family  $( \psi(w) )_{w\in  W\setminus W{\sf MT}_{\leq }(R)W}$ is a basis of $A$. By \Cref{CDA}, the family $( \psi(w) )_{w\in  W\setminus W{\sf mt}_{\preccurlyeq}(S)W}$ is a topological basis of $\hat A.$ Note that $W{\sf mt}_{\preccurlyeq}(S)W=W{\sf MT}_{\leq}(R)W.$ The statement follows.
\end{proof}

\begin{example} The fact that  the  commutative polynomial algebra $$A=\Bbbk[x,y]=\Bbbk \langle x,y\mid xy-yx\rangle$$
is residually nilpotent is obvious but we to want show how to prove this by using \Cref{the_main_theorem}. Take $R=\{xy-yx\}$ and $S=\iota(R).$ Define an order $\leq$ as the deg-lex order corresponding to the order  $x>y$ on letters and $\preccurlyeq$  deg-lex order corresponding to $y>x.$ Then ${\sf MT}_{\leq}(xy-yx)=xy={\sf mt}_{\preccurlyeq}(xy-yx).$ It is easy to see that $R$ and $S$ are Gr\"obner-Shirshov bases in appropriate senses and the assumption of the theorem is satisfied.
\end{example}

\begin{example}\label{ex_r.n} Let $X=\{x_1,x_2,x_3\},$ $r=x_1x_2 +  x_3^2 - x_3$ and $R=\{r \}.$ Consider two orders on $W(x_1,x_2,x_3):$ the deg-lex order $\leq$ corresponding to the order $x_1>x_2>x_3$ and the $\theta$-lex order $\preccurlyeq=\leq_f$ corresponding to the same order on letters and $\theta(x_1)=\theta(x_2)=1 , \theta(x_3)=3.$ Then ${\sf MT}_{\leq }(r)=x_1x_2={\sf mt}_{\preccurlyeq}(r).$  Note that $r$ has no self-compositions and hence $R$ and $\iota(R)$ are Gr\"obner-Shirshov bases. Therefore 
$$A=\Bbbk \langle x_1,x_2,x_3 \mid x_3=x_1x_2+x_3^2 \rangle$$
is residually nilpotent.
\end{example}

This example can be generalized as follows.

\begin{example}
 Let $Z = X \cup Y$, $X = \{x_1,\ldots, x_n\}$ and $Y =\{y_1, \ldots, y_m\}$. Let $u$ be an arbitrary element of $W(X)$ such that  $u$ cannot be decomposed as follows $u= u'u''u'$, with $u' \in W(X)\setminus \{1\}$, $u''\in W(X).$

Let us consider the deg-lext order $\le$ on $W(Z)$ corresponding to the order $x_1>x_2>\cdots > x_n >y_1>y_2>\cdots y_m$ on the letters.
 
Next, let us choose a polynomial $\varphi \in \Bbbk\langle Y \rangle$ such that $v \le u$ (in $W(Z)$) for any monomial $v$ of $\varphi$.
 
Consider now an $\theta$-lex order $\le_\theta$ corresponding the the same order on the letters, $\theta(x_1) = \cdots = \theta(x_n)=1$, and for any $1 \le i \le m$ we put $\theta(y_i) := n_i \in \mathbb{N}$ with $n_i>|u|$.
 
It is clear that ${\sf MT}_\le(r) = {\sf mt}_{\le_\theta}(r)$, for $r = u-\varphi$ and by the construction of $u$ and $\varphi$ it follows that $r$ has no self-compositions, and thus $R$ and $\iota(R)$ are Gr\"obner-Shirshov bases, where $R:=\{r\}.$
 
It implies that an algebra
 \[
  A = \Bbbk \langle x_1,\ldots, x_n, y_1,\ldots, y_m \, | \, u = \varphi \rangle
 \]
  is residually nilpotent.
\end{example}

\section{$2$-acyclic morphisms}

\subsection{Homology of augmented algebras}

The homology of an augmented algebra $A$ with coefficients in an $A$-module $M$  is 
$$H_*(A,M)={\rm Tor}_*^A(M,\Bbbk),$$
and we set $H_*(A):=H_*(A,\Bbbk).$
It is well known that there is an exact sequence
$$0 \longrightarrow H_2(A) \longrightarrow I(A)\otimes_A I(A) \longrightarrow I(A) \longrightarrow H_1(A) \longrightarrow 0,$$
where the middle map is induced by the product. In particular, $$H_1(A)=I(A)/I(A)^2.$$

\begin{lemma}[Hopf's formula]\label{hopf's}
Let $\varphi: A\to B$ be a surjective homomorphism of augmented algebras. Set $\mathfrak{a}={\rm Ker}(\varphi)$ and $I=I(A).$ Then there is an exact sequence
$$H_2(A)\longrightarrow H_2(B) \longrightarrow \frac{\aa\cap I^2   }{I\aa + \aa I} \longrightarrow 0.$$
\end{lemma}
\begin{proof}
It follows from the snake lemma applied to the following diagram
$$
\begin{tikzcd}
0\arrow{r} & H_2(A)\arrow{r}\arrow{d} & I(A)\otimes_A I(A)\arrow{r}\arrow{d} & I(A)^2\arrow{r}\arrow{d} & 0\\
0\arrow{r} & H_2(B)\arrow{r} & I(B)\otimes_A I(B)\arrow{r} & I(B)^2\arrow{r} & 0
\end{tikzcd}
$$
because ${\rm Ker}(I(A)^2\to I(B)^2)=\aa\cap  I(A)^2$ and the sequence $$(I(A)\otimes_A \aa) \oplus (\aa\otimes_A I(A))\to I(A)\otimes_A I(A) \to I(B)\otimes_A I(B) \to 0 $$
is exact. 
\end{proof}

\begin{corollary}
If $A=\Bbbk \langle X \rangle/\mathfrak{r} $ for some ideal $\mathfrak{r}\subseteq I,$ where $I=I(\Bbbk \langle X \rangle),$ then there is an isomorphism
$$H_2(A)\cong \frac{\mathfrak{r}\cap I^2  }{ I\mathfrak{r} +\mathfrak{r} I}.$$
\end{corollary}

\begin{remark}\label{rem_H2}
Let $R\subseteq I=I(\Bbbk \langle X \rangle)$ be a set and 
$A=\Bbbk \langle X \mid R \rangle.$ Then any element of $H_2(A)$ can be presented as a  linear combination of the relations 
$$\sum_{i=1}^n \alpha_i r_i \in (R)\cap I^2, \hspace{1cm} \alpha_i\in \Bbbk, r_i\in R $$
without summands given by products $pr,rp,$ where $r\in R,p\in \Bbbk \langle X \rangle$ because they lie in $(R)I+I(R).$ 

\end{remark}

\subsection{Paraequivalences}
A morphism of augmented algebras $A\to B$ is called {\it para-equivalence} if it induces an isomorphism $A/I(A)^n \cong B/I(B)^n$ for each $n.$ It is easy to see that a paraequivalence $A\to B$ induces an isomorphism of completions $\hat A\cong \hat B.$ If $A$ and $B$ are of finite type, then the opposite also holds by  \Cref{th_completion_is_complete}. 
For example, the map $A\to \hat A$ is a paraequivalence for an augmented algebra of finite type.

\subsection{2-acyclic morphisms}

A morphism of augmented algebras $A\to B$ is called {\it $n$-acyclic} if it induces an isomorphism $H_i(A)\cong H_i(B)$ for $i<n$ and an epimorphism $H_n(A)\twoheadrightarrow H_n(B).$ Here we give an analogue of Stalling's theorem \cite{St65} with a similar proof.

\begin{proposition}[Stalling's theorem]
A $2$-acyclic morphism is a paraequivalence. 
\end{proposition}
\begin{proof} Prove that $A/I(A)^n\to B/I(B)^n$ is an isomorphism. 
The proof is by induction on $n$. For $n=1$ it is obvious. For $n=2$ it is equivalent to the isomorphism $H_1(A)\cong H_1(B).$ Assume $n>2.$ \Cref{hopf's} implies that the rows of the diagram 
$$
\begin{tikzcd}
H_2(A) \arrow{r} \arrow{d} & H_2(A/I(A)^{n-1}) \arrow{r} \arrow{d} & A/I(A)^{n} \arrow{r} \arrow{d} & 0 \\
H_2(B) \arrow{r} & H_2(B/I(B)^{n-1}) \arrow{r} & B/I(B)^{n} \arrow{r} & 0
\end{tikzcd}
$$
are exact. By induction hypothesis, we obtain that the middle vertical map is an isomorphism. Then a simple diagram chasing shows that $A/I(A)^n\to B/I(B)^n$ is surjective.
\end{proof}

\begin{lemma}[$2$-acylic surjections] A surjective
morphism $A\epi B$ is $2$-acylic if and only if its kernel $\aa={\rm Ker}(A\to B)$ satisfies
$\aa\subseteq I^2$ and $\aa=I\aa+\aa I.$
\end{lemma}
\begin{proof}
Since $ I(A)/I(A)^2\cong H_1(A)\to H_1(B)\cong I(B)/I(B)^2$ is an isomorphism, we obtain $\aa\subseteq I(A)^2.$ Then \Cref{hopf's} implies $\aa=I\aa+\aa I.$
\end{proof}

\subsection{Standard 2-acyclic morphisms}

If $A$ and $B$ are two algebras we denote by $A*B$ their free product i.e., the coproduct in the category of algebras. If $A$ and $B$ are augmented, then by the universal property, the augmentations induce an augmentation $A*B\to \Bbbk.$ This augmented algebra $A*B$ is the coproduct of $A$ and $B$ in the category of augmented algebras. 

Let $A$ be an augmented algebra and let $X$ be a set, which we call a set of variables. We define a $A$-polynomial as an element of $A*\Bbbk \langle X \rangle.$ Consider the morphism
$$A*  \Bbbk \langle X \rangle \longrightarrow \Bbbk \langle X \rangle/I^2,$$
which sends $I(A)$ to zero. An  $A$-polynomial $p\in A*\Bbbk \langle X \rangle$ is called {\it acyclic}, if it is in the kernel of this map. Let $$\mathcal P=(p_x)_{x\in X}$$ be a family of $A$-acyclic polynomials indexed by the set of variables. We denote by $\mathfrak{r}_\mathcal{P}$ the ideal of $A* \Bbbk \langle X \rangle $ generated by the set $\{x-p_x\mid x\in X\}$ and define an algebra $A_\mathcal P$ as follows
$$A_\mathcal{P}=(A* \Bbbk \langle X \rangle )/\mathfrak{r}_\mathcal{P}.$$ 
We denote by $w_\mathcal{P} : A \longrightarrow A_\mathcal{P}$
the obvious homomorphism. 

\begin{proposition}
\label{prop_stand_2_ac} 
The morphism $w_\mathcal{P}:A\to A_\mathcal{P}$  is 2-acyclic. 
\end{proposition}
\begin{proof} It is easy to see that $I(A*B)/I(A*B)^2=I(A)/I(A)^2\oplus I(B)/I(B)^2$ for arbitrary $A,B.$ Set $F=\Bbbk\langle X \rangle.$ Consider the map $\varphi: I(A*F) \to I(A)/I(A)^2 \oplus I(F)/I(F)^2.$ Then  $\varphi(p_x)=(a_x+A^2,0)$ for some $a_x\in A.$ Set $q_x=p_x-a_x.$ Note that $q_x\in I(A*F)^2.$ Then the ideal $\mathfrak{r}_\mathcal{P}$ is generated by the relations $x=a_x+q_x,$ where $a_x\in A$ and $q_x\in I(A*F)^2.$ Therefore, $\mathfrak{r}_\mathcal{P}+I(A*F)^2=J+I(A*F)^2,$ where $J$ is generated by relations 
$x-a_x.$ Since $(A*F)/J=A,$ we obtain  
$A_\mathcal{P}/I(A_\mathcal{P})^2 = (A*F)/(J+I(A*F)^2)=A/I(A)^2.$ 
Therefore, the map $H_1(A)\to H_1( A_\mathcal{P})$ is an isomorphism. 

Let $A=\Bbbk \langle Y \mid R \rangle$ for some set $Y$ and some set $R\subseteq I( \Bbbk \langle Y \rangle).$ Then $$A_\mathcal{P}=\Bbbk \langle X\sqcup Y \mid \tilde{R}\cup \{x-p_x \mid x\in X\}  \rangle,$$
where $\tilde{R}$ 
is the image of $R$ in $\Bbbk \langle X\sqcup Y \rangle.$ Here we identify the second 
homology with the Hopf's 
formula. Consider an element $\theta\in H_2(A_\mathcal{P})$ 
and  presented it  as a linear 
combination of the relations 
$$\theta=\sum_{r\in R}\alpha_r r+\sum_{x\in X} \beta_x (x-p_x)  $$
(see Remark \ref{rem_H2}). The image of $\theta$ in $F/I(F)^2$ equals to zero because $\theta$ lies in $I(\Bbbk \langle X\cup Y\rangle )^2$. On the other hand it equals to $\sum_{x\in X}\beta_xx.$ Therefore $\beta_x=0$ for any $x\in X.$ It follows that $\theta=\sum_{r\in R}\alpha_r r.$ Then it comes from $H_2(A).$ Thus $H_2(A)\to H_2(A_\mathcal{P})$ is an epimorphism. 
\end{proof}

The morphisms $w_\mathcal{P}:A\to A_\mathcal{P}$ will be called {\it standard $2$-acyclic morphisms}.

\section{Parafree augmented algebras}

An augmented algebra $A$ is called {\it parafree} if it is residually nilpotent and there is a paraequiavence from a free algebra $\Bbbk \langle X \rangle \to A,$ where $X$ is finite. In paticular, the completion of a parafree algebra is the algbra of non-commutative power series $\hat A\cong  \Bbbk \lb X \rb.$

There is a standard way to construct parafree algebras. Take a family $\mathcal P=(p_y)_{y\in Y}$ of acyclic $\Bbbk \langle X \rangle$-polynomials. This means that $p_y\in {\rm Ker} (\Bbbk \langle X \sqcup Y \rangle \to \Bbbk \langle Y \rangle /I^2).$
Then, by  \Cref{prop_stand_2_ac}, the map $$\Bbbk \langle X \rangle \longrightarrow \Bbbk \langle X \rangle_{\mathcal P}$$ is 2-acyclic and hence it is a 
paraequivalence. The algebra $\Bbbk \langle X \rangle_{\mathcal P}$  need not to be parafree, because it is not always residually nilpotent. However, the quotient 
$${\sf Para}( \mathcal P)=\Bbbk \langle X \rangle_{\mathcal P}/I^\omega$$
is parafree, where $I^\omega=\bigcap I^n.$ Such augmented algebras will be called standard parafree algebras. If $Y=\{y_1,\dots,y_M\}$ is finite, then ${\sf Para}(\mathcal{P})$ is finitely generated. In this case we will use notations $p_i=p_{y_i}$ and 
${\sf Para}(p_1,\dots,p_M):={\sf Para}(\mathcal{P}).$
Then 
$${\sf Para}(p_1,\dots,p_M)=\Bbbk\langle x_1,\dots,x_N,y_1,\dots,y_M\mid y_1=p_1,\dots,y_M=p_M \rangle /I^\omega.$$

\begin{example} Let $X=\{x_1,x_2\}$ and $Y=\{y\}.$ Then 
${\sf Para}( x_1x_2+y^2 )= \Bbbk \langle x_1,x_2,y \mid y=x_1x_2+y^2 \rangle/I^\omega.$
 Then, by \Cref{ex_r.n}, we have  $I^\omega=0.$ Therefore 
$${\sf Para}( x_1x_2+y^2 )= \Bbbk \langle x_1,x_2,y \mid y=x_1x_2+y^2 \rangle.$$
\end{example}

\section{The main example}

Let $X=\{x_1,x_2\}$ and $Y=\{y_1,y_2\}.$ The elements $x_1x_2+y_1^2,x_2x_1+y_2^2$ are acyclic $\Bbbk \langle X \rangle$-polynomials. Consider the following parafree augmented  algebra
$$A:={\sf Para}(x_1x_2+y_1^2,x_2x_1+y_2^2).
$$
Set
$$r_1:=x_1x_2+y_1^2-y_1, \hspace{1cm} r_2=x_2x_1+y_2^2-y_2, $$
$$r_3:=x_1y_2-y_1x_1, \hspace{1cm}  r_4:=x_2y_1-y_2x_2$$
and  $B:= \Bbbk \langle X \rangle_{\mathcal{P}} =\Bbbk \langle x_1,x_2,y_1,y_2\mid  r_1, r_2 \rangle.$
Then $A=B/I^\omega.$ 

\begin{lemma}
The parafree algebra $A$ has the following presentation
$$A=\Bbbk \langle x_1,x_2,y_1,y_2\mid   r_1,r_2,r_3,r_4 \rangle.$$
Moreover, if we take the deg-lex order on $W(x_1,x_2,y_1,y_2)$ such that $x_1>x_2>y_1>y_2,$ then $R=\{r_1,r_2,r_3,r_4\}$ is a Gr\"obner-Shirshov basis. 
\end{lemma}
\begin{proof}
It is enough to prove that $I^\omega(B)=(r_3,r_4).$ Prove that $ r_3,r_4 \in I^\omega(B).$ Note that
\begin{equation*}
\begin{split}
r_3&=x_1(x_2x_1+y_2^2)-(x_1x_2+y_1^2)x_1  \\
&= x_1y_2^2-y_1^2x_1\\
&=r_3y_2-y_1r_3.
\end{split}    
\end{equation*}
Using the equation $r_3=r_3y_2-y_1r_3 ,$ it is easy to prove by induction that $r_3\in I^n(B)$ for any $n.$ Similarly $r_4\in I^n(B)$ for any $n.$ Then $r_3,r_4\in I^\omega(B).$ In order to prove that $I^\omega(B)=(r_3,r_4),$ it is enough to check that $\tilde A:=B/(r_3,r_4)$ is residually nilpotent. 

Prove that $\tilde A$ is residually nilpotent using \Cref{the_main_theorem}. Set $R=\{ r_1,r_2,r_3,r_4 \}.$ Then $\tilde A=\Bbbk \langle x_1,x_2,y_1,y_2 \rangle/(R).$ Consider two orders: $\leq$ is the deg-lex order such that $x_1>x_2>y_1>y_2$ and $\preccurlyeq$ is the $\theta$-lex order such that $\theta(x_1)=\theta(x_2)=1,$ $\theta(y_1)=\theta(y_2)=3$ and $y_1>y_2>x_1>x_2.$ Then $${\sf MT}_{\leq}(r_1)=x_1x_2= {\sf mt}_{\preccurlyeq}(r_1), \hspace{1cm} {\sf MT}_{\leq}(r_2)=x_2x_1= {\sf mt}_{\preccurlyeq}(r_2),$$  
$${\sf MT}_{\leq}(r_3)=x_1y_2= {\sf mt}_{\preccurlyeq}(r_3), \hspace{1cm} {\sf MT}_{\leq}(r_4)=x_2y_1= {\sf mt}_{\preccurlyeq}(r_4).$$
Straightforward computations show that $R$ and $\iota(R)$  are  Gr\"obner-Shirshov bases in appropriate senses.
\Cref{the_main_theorem} implies that $\tilde A$ is residually nilpotent. Therefore $I^\omega(B)=(r_3,r_4),$ and hence,  $A=\tilde A.$ 
\end{proof}

Further we will identify elements of $\Bbbk \langle x_1,x_2,y_1,y_2 \rangle$ with their images in $A.$

Denote by $\mathcal B$ the set of words in $W(x_1,x_2,y_1,y_2)$ that do not contain $x_1x_2,$ $x_2x_1,$ $x_1y_2,$ $x_2y_1$ as subwords. Then, by the CD-lemma, $(w+\mathfrak{r})_{w\in \mathcal B}$ is a basis of $A,$ where $\mathfrak{r}= (r_1,r_2,r_3,r_4).$  For each word $w\in \mathcal B$ and each letter $l\in \{x_1,x_2,y_1,y_2\}$ we set 
$$c_l(w):= \begin{cases} u, & \text{if } w=lu\ \  \text{ for } u\in \mathcal B  \\ 
0, & \text{else.} 
\end{cases} $$  
In particular, $c_l(1)=0.$
Extend this to a $\Bbbk$-linear map $c_l:A\to A.$ Then for any $a$ we have
\begin{equation}\label{eq_c0}
  a=\varepsilon(a)+ x_1c_{x_1}(a)+x_2c_{x_2}(a)+y_1c_{y_1}(a)+y_2c_{y_2}(a).  
\end{equation}
If $a\in A$ we denote by $c_l a$ and $ac_l$ maps $A\to A$ given by compositions  $(c_la)(x)= c_l(ax) $ and $(ac_l)(x)=ac_l(x).$ Using this functional language we can rewrite the equation \eqref{eq_c0} as 
$$
    {\sf id}=\varepsilon+ x_1c_{x_1}+x_2c_{x_2}+y_1c_{y_1}+y_2c_{y_2}.
$$
\begin{lemma}\label{lemma_eqations_for_c} The following equations for $c_{x_1},c_{x_2},c_{y_1},c_{y_2}$ hold
$$\begin{array}{rclcrcl}
c_{x_1} x_1&=&\varepsilon+ x_1c_{x_1}+y_1c_{y_1}, & \hspace{0.5cm} &  c_{x_1}x_2&=&0 , \\
c_{x_2}x_1&=&0 &&  c_{x_2}x_2&=&\varepsilon+ x_2c_{x_2}+y_2c_{y_2}, \\
c_{y_1}x_1&=&-(y_1-1)c_{x_2}+x_1 c_{y_2}, && c_{y_1}x_2&=&0, \\
c_{y_2}x_1&=&0 && c_{y_2}x_2&=&-(y_2-1)c_{x_1}+x_2 c_{y_1}, 
\end{array}
$$
$$
\begin{array}{rcl}
   c_{y_i}y_i&=&{\sf id},\\
   c_{l}y_i&=&0 \text{ \ for\ } l\ne y_i,
\end{array}
$$
where $i=1,2.$
\end{lemma}
\begin{proof}
First we prove the equations for $c_lx_1.$  It is enough to prove that the equations holds on all $w\in \mathcal B.$ If $w=1,$  it is obvious. Assume that $w\ne 1$ and denote by $l$ the first letter of $w.$ Consider the following cases. 

Case 1: $l=x_1$ or $l=y_1$. In this case $x_1w\in \mathcal B$ and $c_{x_1}(x_1w)=w=x_1c_{x_1}(w)+y_1c_{y_1}(w).$ We also have $0=c_{x_2}(x_1w)=c_{y_1}(x_1w)=c_{y_2}(x_1w)$ and $0=c_{x_2}(w)=c_{y_2}(w).$

Case 2: $l=x_2.$ In this case $w=x_2u$ for some $u\in \mathcal{B}$ and we have $x_1w \equiv x_1x_2u\equiv -y_1^2u+y_1u.$ Note that $y_1^2u,y_1u\in \mathcal{B}.$ Then $c_{y_1}(x_1w)=-y_1u+u=-(y_1-1)c_{x_2}(w)$ and $0=c_{x_1}(x_1w)=c_{x_2}(x_1w)=c_{y_2}(x_1w_1)=c_{x_1}(w)=c_{y_2}(w).$

Case 3: $l=y_2.$ In this case $w=y_2^nu$ for some $u\in \mathcal B$ whose the first letter is not $y_2$ and some  $n\geq 1.$  Consider two sub-cases here. 

\begin{itemize}
    \item[] Case 3.1: the first letter of $u$ is not $x_2.$ In this case $x_1w=x_1y_2^nu= y_1^nx_1u$ and $y_1^nx_1u \in  \mathcal{B}.$ Then $c_{y_1}(x_1w)=y_1^{n-1}x_1u=x_1y_2^{n-1}u=x_1 c_{y_2}(w).$ We also have $c_{s}(x_1w)=0$ for $s=x_1,x_2,y_2$ and $c_{t}(w)=0$ for $t=x_1,x_2,y_1.$
    
    \item[] Case 3.2: the first letter of $u$ is $x_2.$ Then $u=x_2v.$ In this case $x_1w=x_1 y_2^n x_2 v=-y_1^n (y_1^2-y_1)v=-(y_1^2-y_1)y_1^n v.$ Note that $y_1^kv\in \mathcal B$ for $k=n+1,n+2,$ and hence,  $c_{y_1}(x_1w)=-y_1^{n-1}(y_1^2-y_1)v=  y_1^{n-1} x_1x_2v=x_1y_2^{n-1}x_2 v=x_1c_{y_2}(w).$ We also have $c_{s}(x_1w)=0$ for $s=x_1,x_2,y_2$ and $c_{t}(w)=0$ for $t=x_1,x_2,y_1.$ 
\end{itemize} 
So we proved the equations for $c_lx_1.$ The equations for $c_lx_2$ are similar. The equations for $c_ly_i$ follow from the fact that  $w\in \mathcal B$ implies $y_iw\in \mathcal B.$
\end{proof}

Now we will work with right modules. A homomorphism of {\it right} modules $A\to A$ is a {\it left} multiplication $x\mapsto ax$ on some element $a\in A.$ We will also identify elements of $A$ with such homomorphisms $a:A\to A.$ Elements of $A^n$ will be written as columns. A homomorphism of right modules $A^n\to A^m$ will be written as a $m\times n$-matrix over $A.$ Composition then corresponds to the matrix multiplication $(a_{i,j})_{i,j} (b_{j,k})_{j,k}=(\sum_j a_{i,j}b_{j,k})_{i,k}.$ 

More generally, if we have some $\Bbbk$-module $M,$ we describe $\Bbbk$-linear maps  $M^n\to M^m$ as $m\times n$-matrices $(f_{i,j})_{i,j}$ of $\Bbbk$-linear maps $f_{i,j}:M\to M.$

Consider the following chain complex $P_\bullet$ such that $P_i=0$ for $i<0,$ $P_0=A,$ $P_i=A^4$ for $i> 0$
$$\dots \xrightarrow{d_4}  A^4 \xrightarrow{d_3}  A^4 \xrightarrow{d_2} A^4 \xrightarrow{d_1}  A \to 0 \to \dots   $$
whose differentials are described as follows: 
$d_1=(x_1,x_2,y_1,y_2)$
and
$$d_{2+2n}= 
\left( \begin{matrix}
x_2   &  0     & -y_2   & 0    \\
0     & x_1    & 0      & -y_1 \\
y_1-1 & 0      & x_1    & 0    \\
0     & y_2-1  & 0      & x_2  \\
\end{matrix} \right), \hspace{0.5cm} d_{3+2n}=
\left( \begin{matrix}
x_1     & 0       & y_1     & 0   \\
0       & x_2     & 0       & y_2 \\
-y_2+1  & 0       & x_2     & 0   \\
0       & -y_1+1  & 0       & x_1 
\end{matrix} \right)$$
for $n\geq 0.$ A straightforward computation show that $d_id_{i+1}=0.$ 

\begin{lemma} $P_\bullet$ is a projective resolution of the trivial right module $\Bbbk$ over $A.$
\end{lemma}
\begin{proof} In order to prove that $P_\bullet$ is a resolution, it is enough to prove that the augmented resolution $P'_\bullet=(\dots \to P_1\to P_0\overset{\varepsilon}\to \Bbbk \to 0 \to \dots)$ is contractible as a complex of $\Bbbk$-modules. Consider the following $\Bbbk$-linear maps $h_i:P_i\to P_{i+1}$ and $h_{-1}:\Bbbk \to P_0$ given by $h_{-1}(1)=1,$  $h_0=(c_{x_1},c_{x_2},c_{y_1},c_{y_2})^{\top} $
$$ h_{1+2n}=\left( 
\begin{smallmatrix}
c_{x_2} &  0 & 0 & 0\\
0 & c_{x_1}  & 0 & 0\\
-c_{y_2} & 0 & 0 & 0\\
0 & -c_{y_1} & 0 & 0
\end{smallmatrix}
\right), \hspace{1cm} 
h_{2+2n}=\left( 
\begin{smallmatrix}
c_{x_1} &  0 & 0 & 0\\
0 & c_{x_2}  & 0 & 0\\
c_{y_1} & 0 & 0 & 0\\
0 & c_{y_2} & 0 & 0
\end{smallmatrix}
\right)$$
for $n\geq 0.$ Using \Cref{lemma_eqations_for_c} one can check
$$h_{-1}\varepsilon + d_1 h_0={\sf id}_{A},  $$
$$h_{i-1}d_{i} + d_{i+1} h_{i}={\sf id}_{A^4}$$
for $i=1,2,3.$ This implies that the augmented resolution $P'_\bullet$ is contractible. 
\end{proof}

\begin{remark}
 The contracting homotopy for this resolution was hinted by results of \cite{JW09}, \cite{Sc06} about Anick's resolution.
\end{remark}

 Note that ${\rm Im}(d_1:P_1\to P_0)=I(A).$ Denote by $\tilde d_1:A^4\epi  I(A)$ the restriction of the map $d_1.$  Then $\tilde d_1$ induces an isomorphism 
$I(A)\cong  {\rm Coker}(d_{2+2n}).$
It follows that there exists a map $\alpha: I(A)\to P_{2+2n}$ such that $d_{3+2n}=\alpha \tilde d_1.$ Then we have an exact sequence
\begin{equation}\label{eq_ext}
0 \to I(A)\to P_{2+2n} \to \dots \to P_1 \to P_0 \to \Bbbk \to 0 
\end{equation}
for any $n\geq 0.$ This exact sequence represents an element in ${\rm Ext}^{3+2n}_A(\Bbbk,I(A))=H^{3+2n}(A,I(A)).$

\begin{lemma}\label{lemma_ext} The exact sequence \eqref{eq_ext} represents a non-trivial element of $H^{3+2n}(A,I(A))$ for any $n\geq 0.$ In particular, 
$$H^{3+2n}(A,I(A))\ne 0$$
for any $n\geq 0.$
\end{lemma}
\begin{proof}
The diagram 
$$
\begin{tikzcd}
P_{3+2n}\arrow{r}{d_{3+2n}} \arrow{d}{d'_1} & P_{2+2n}\arrow{r}\arrow[equal]{d} & \dots\arrow{r} & P_1\arrow{r}\arrow[equal]{d} & P_0\arrow{r}\arrow[equal]{d} & \Bbbk \arrow[equal]{d} \\ 
I(A)\arrow{r}{\alpha} & P_{2+2n}\arrow{r} & \dots \arrow{r} & P_1\arrow{r} & P_0\arrow{r} & \Bbbk 
\end{tikzcd}
$$
shows that the exact sequence \eqref{eq_ext} corresponds to $\tilde d_1$ in ${\rm Hom}(P_{3+2n},I(A)),$ if we present $H^{3+2n}(A,I(A))$ as the  homology of the complex ${\rm Hom}(P_\bullet,I(A)).$ If we use the isomorphism ${\rm Hom}_A(A^4,I(A))\cong I(A)^4$ we obtain that the complex $ {\rm Hom}(P_\bullet,I(A))$ is isomorphic to the complex
$$\dots  \xleftarrow{\ d_4^{\top}} I(A)^4 \xleftarrow{\ d_3^{\top}} I(A)^4 \xleftarrow{\ d_2^{\top}}  I(A)^4
\xleftarrow{\ d_1^{\top}}  I(A) \leftarrow 0,$$
where $d_i^\top$ denotes transposed matrix $d_i$ considered as a map $I(A)^n\to I(A)^m.$ Moreover, the element $d'_1:P_{3+2n}\to I(A)$ corresponds to the element  $(x_1,x_2,y_1,y_2)^\top\in I(A)^4.$ Then it is enough to prove that $(x_1,x_2,y_1,y_2)^\top\notin {\rm Im} (d_{3+2n}^\top),$
where 
$$d_{3+2n}^\top=
\left( \begin{smallmatrix}
x_1 & 0   & -y_2+1 & 0     \\
0   & x_2 & 0      & -y_1+1\\
y_1 & 0   & x_2    & 0     \\
0   & y_2 & 0      & x_1
\end{smallmatrix}
\right) : I(A)^4 \to I(A)^4.
$$
Assume the contrary that there exists $v=(v_1,v_2,v_3,v_4)\in I(A)^4$ such that $d_{3+2n}^\top(v)=(x_1,x_2,y_1,y_2)^\top.$ This implies  $x_1v_1+(-y_2+1)v_3=x_1.$ We apply $c_{y_2}$ to both sides of this equation. By Lemma \ref{lemma_eqations_for_c},  we obtain $c_{y_2}(x_1v_1)=0$ and $c_{y_2}(y_2v_3)=v_3.$ Therefore, the equation $x_1v_1+(-y_2+1)v_3=x_1$ implies $v_3=c_{y_2}(v_3).$ It follows that $v_3=0.$  Then $x_1=x_1v_1$ and $v_1\in I(A).$ Then by induction we obtain $x_1\in I(A)^n$ for any $n.$ Since $A$ is residually nilpotent, this implies $x_1=0.$ This is a contradiction.
\end{proof}

\begin{theorem}
The algebra $A$ is a finitely generated parafree algebra of infinite cohomological dimension.
\end{theorem}
\begin{proof} 
This follows from  \Cref{lemma_ext} and the definition of $A$.
\end{proof}


\begin{thebibliography}{20}

\bibitem[Ba67a]{Ba67a} G. Baumslag: Some groups that are just about free, Bull. Amer. Math. Soc. 73 (1967), 621–622.

\bibitem[Ba67b]{Ba67b} G. Baumslag: Groups with the same lower central sequence as a relatively free group. I. The
groups, Trans. Amer. Math. Soc. 129 (1967), 308-–321.

\bibitem[Ba68]{Ba68} G. Baumslag: More groups that are just about free, Bull. Amer. Math. Soc. 74 (1968), 752-
–754

\bibitem[Ba69]{Ba69} G. Baumslag: Groups with the same lower central series as a relatively-free group, II, Trans.
Amer. Math. Soc. 142 (1969), 507-–538.

\bibitem[BC14]{BC14} L.A. Bokut and Y. Chen, Gr\"obner--Shirshov basis and their calculation, {\it Bull. Math. Sci.} {\bf 4}, (2014), 325--395.

\bibitem[Bou77]{Bou77} A.K. Bousfield: Homological localization towers for groups and $\pi$-modules, Mem. Amer. Math.
Soc, no. 186, 1977

\bibitem[Buch70]{Buchberger} B. Buchberger, An algorithmical criteria for the solvability of algebraic systems of equations, {\it Aequationes Math.,} {\bf 4}, 374--383 (1970).

\bibitem[Bour72]{Bour72} N. Bourbaki, Commutative Algebra, Chapter 1-7, Hermann, 1972.



\bibitem[Co87]{Co87} T. Cochran: Link concordance invariants and homotopy theory, Invent. Math. 90 (1987),
635–645.


\bibitem[CO98]{CO98} T. Cochran and K. Orr: Stability of lower central series of compact 3-Manifold groups, Topology
37 (1998), 497–526.

\bibitem[GH98]{GH98} L. Gerritzen and R. Holtkamp, On Gr\"obner bases of noncommutative power series, \textit{Indag. Mathem., N.S.,} \textbf{9}(4), (1998) 503--519.

\bibitem[H02]{H02} L. Hellstr\"om, {\it The Diamond Lemma for Power Series Algebras} (doctorate thesis), 2002, xviii+228 pp.; ISBN 91-7305-327-9;  \newline\texttt{http://abel.math.umu.se/lars/diamond/thesis.pdf}

\bibitem[Hir]{Hir} H. Hironaka: Resolution of singularities of an algebraic variety over a Field
of characteristic zero, Ann. of Math. (2) 79 (1964), 109—203 (part I) and
205—326 (part II).

\bibitem[IM18]{IM18} S. O. Ivanov, R. Mikhailov: On discrete homology of a free pro-$p$-group. Compositio Math. 154 (2018) 2195-2204.

\bibitem[IM19]{IM19} S. O. Ivanov, R. Mikhailov: A finite $\mathbb Q$-bad space. Geom. Topol. 23 (2019) 1237-1249.

\bibitem[IMZ]{IMZ} S. O. Ivanov, R. Mikhailov, A. Zaikovski: Homological properties of parafree Lie algebras. Journal of Algebra, v. 560, (2020), pp. 1092-1106.

\bibitem[JW09]{JW09} M. J\"ollenbeck and V. Welker, 
\emph{Minimal Resolutions via Algebraic Discrete Morse Theory}, 
Memoirs Amer. Math. Soc. \textbf{197} (2009), no.923.

\bibitem[Mo94]{Mo94} T. Mora, An introduction to commutative and noncommuative Gr\"obner Bases, {\it Theor. Comp. Sci.} {\bf 134} (1994) 131--173.

\bibitem[Sc06]{Sc06} E. Sk\"oldberg:  Morse theory from an algebraic viewpoint, {\it Trans. Amer. Math. Soc.} \textbf{358}(1) (2006) 115--129.

\bibitem[St65]{St65} J. Stallings:  Homology and central series of groups, J. Algebra 2 (1965), 170–181.

\bibitem[St68]{St68} J. Stallings: On torsion-free groups with infinitely many ends, Ann. Math. 88, (1968) 312-334.

\bibitem[Sw69]{Sw69} R. Swan: Groups of cohomological dimension one, J.Algebra 12, (1969), 585-610.

\end{thebibliography}
\end{document}